\documentclass[11pt]{amsart}
\usepackage{babel}
\usepackage[all]{xy}
\usepackage{amssymb}
\usepackage{enumitem}
\usepackage{amsthm}
\usepackage{colortbl}
\usepackage[dvipsnames]{xcolor}
\usepackage{hyperref}
\usepackage{amsmath}
\usepackage{amscd,enumitem}
\usepackage{verbatim}
\usepackage{eurosym}
\usepackage{float}
\usepackage{cleveref}
\usepackage{color}
\usepackage{rotating}
\usepackage{dcolumn}
\usepackage[mathscr]{eucal}
\usepackage[all]{xy}
\usepackage{bbm}
\usepackage{makecell}
\usepackage{wasysym}
\usepackage[margin=1in]{geometry} 
\usepackage{tikz}
\usepackage{graphicx}
\newcommand\myhdots{\hbox to 1.5em{.\hss.\hss.}}
\usetikzlibrary{automata,arrows,positioning,calc}
\DeclareFontFamily{U}{wncy}{}
    \DeclareFontShape{U}{wncy}{m}{n}{<->wncyr10}{}
    \DeclareSymbolFont{mcy}{U}{wncy}{m}{n}
    \DeclareMathSymbol{\Sh}{\mathord}{mcy}{"58} 
\newtheorem*{thm*}{Theorem}
\newtheorem*{conj*}{Conjecture}

\newtheorem*{remark}{Remark}

\newtheorem{theorem}{Theorem}[section]

\newtheorem{lemma}[theorem]{Lemma}

\newtheorem{proposition}[theorem]{Proposition}

\newtheorem{corollary}[theorem]{Corollary}

\theoremstyle{definition}
\newtheorem{definition}[theorem]{Definition}
\newtheorem{example}[theorem]{Example}

\newtheorem*{question}{Question}

\newcommand{\NN}{\mathbb{N}}

\newcommand{\ceil}[1]{\lceil #1 \rceil}

\newcommand{\floor}[1]{\lfloor #1 \rfloor}

\numberwithin{equation}{section}

\renewenvironment{quote}%
  {\list{}{\leftmargin=0.25in}\item[]}%
  {\endlist}

  \usepackage{tikz}

\hypersetup{
    colorlinks=true,
    linkcolor=blue,
    citecolor=red
    }

\usetikzlibrary{calc}
\tikzset{vtx/.style={circle, fill, inner sep=1.5pt}}
\tikzset{openvtx/.style={circle, draw, inner sep=1.5pt}}

\usepackage[dvipsnames]{xcolor}
\usepackage{enumitem}
\usepackage{tikz}
\usetikzlibrary{automata,arrows,positioning,calc}
\DeclareFontFamily{U}{wncy}{}
    \DeclareFontShape{U}{wncy}{m}{n}{<->wncyr10}{}
    \DeclareSymbolFont{mcy}{U}{wncy}{m}{n}
    \DeclareMathSymbol{\Sh}{\mathord}{mcy}{"58} 

\usetikzlibrary{calc}
\tikzset{vtx/.style={circle, fill, inner sep=1.5pt}}
\tikzset{smallvtx/.style={circle, fill, inner sep=1.5pt, scale=0.9}}
\tikzset{openvtx/.style={circle, draw, inner sep=1.5pt}}

\usetikzlibrary{decorations.pathmorphing, shapes, patterns}

\usepackage{diagbox}

\title{Counting numerical semigroups by Frobenius number, multiplicity, and depth}
\author{Sean Li}
\address{Department of Mathematics, Massachusetts Institute of Technology, Cambridge, MA 02139, USA}
\email{seanjli@mit.edu}
\makeatletter
\renewcommand{\paragraph}[1]{%
  \par
  \addvspace{\medskipamount}
  \textit{#1\@addpunct{.}}\enspace\ignorespaces
}
\makeatother

\makeatletter
\renewcommand{\subparagraph}[1]{%
  \par
  \addvspace{\medskipamount}

  \noindent
  \textbf{#1\@addpunct{.}}\enspace\ignorespaces
}
\makeatother

\usepackage{comment}

\usepackage{epigraph}
\usepackage{stmaryrd}
\usetikzlibrary{patterns}
\usetikzlibrary{decorations.pathmorphing}

\newcommand{\fr}{\operatorname{Fr}}
\newcommand{\ah}{\operatorname{st}}
\newcommand{\MED}{\operatorname{MED}}

\newcommand{\eps}{\varepsilon}
\renewcommand{\hom}{\operatorname{hom}}
\newcommand{\cK}{\mathcal{K}}

\newcommand{\mult}{\operatorname{Mult}}

\usepackage[textwidth=25mm]{todonotes}

\newcommand{\vocab}{\emph}

\usepackage{thmtools}

\newlength{\hatchspread}
\newlength{\hatchthickness}
\newlength{\hatchshift}
\newcommand{\hatchcolor}{}
\tikzset{hatchspread/.code={\setlength{\hatchspread}{#1}},
         hatchthickness/.code={\setlength{\hatchthickness}{#1}},
         hatchshift/.code={\setlength{\hatchshift}{#1}},
         hatchcolor/.code={\renewcommand{\hatchcolor}{#1}}}
\tikzset{hatchspread=3pt,
         hatchthickness=0.4pt,
         hatchshift=0pt,
         hatchcolor=black}
\pgfdeclarepatternformonly[\hatchspread,\hatchthickness,\hatchshift,\hatchcolor]
   {custom north west lines}
   {\pgfqpoint{\dimexpr-2\hatchthickness}{\dimexpr-2\hatchthickness}}
   {\pgfqpoint{\dimexpr\hatchspread+2\hatchthickness}{\dimexpr\hatchspread+2\hatchthickness}}
   {\pgfqpoint{\dimexpr\hatchspread}{\dimexpr\hatchspread}}
   {
    \pgfsetlinewidth{\hatchthickness}
    \pgfpathmoveto{\pgfqpoint{0pt}{\dimexpr\hatchspread+\hatchshift}}
    \pgfpathlineto{\pgfqpoint{\dimexpr\hatchspread+0.15pt+\hatchshift}{-0.15pt}}
    \ifdim \hatchshift > 0pt
      \pgfpathmoveto{\pgfqpoint{0pt}{\hatchshift}}
      \pgfpathlineto{\pgfqpoint{\dimexpr0.15pt+\hatchshift}{-0.15pt}}
    \fi
    \pgfsetstrokecolor{\hatchcolor}
    \pgfusepath{stroke}
   }

\begin{document}
\maketitle

\begin{abstract}
  In 1990, Backelin showed that the number of numerical semigroups with Frobenius number $f$ approaches $C_i \cdot 2^{f/2}$ for constants $C_0$ and $C_1$ depending on the parity of $f$. In this paper, we generalize this result to semigroups of arbitrary depth by showing there are $\floor{(q+1)^2/4}^{f/(2q-2)+o(f)}$ semigroups with Frobenius number $f$ and depth $q$. More generally, for fixed $q \geq 3$, we show that, given $(q-1)m < f < qm$, the number of numerical semigroups with Frobenius number $f$ and multiplicity $m$ is
  \[\left(\left\lfloor \frac{(q+2)^2}{4} \right\rfloor^{\alpha/2} \left \lfloor \frac{(q+1)^2}{4} \right\rfloor^{(1-\alpha)/2}\right)^{m + o(m)}\]
where $\alpha = f/m - (q-1)$. Among other things, these results imply Backelin's result, strengthen bounds on $C_i$, characterize the limiting distribution of multiplicity and genus with respect to Frobenius number, and resolve a recent conjecture of Singhal on the number of semigroups with fixed Frobenius number and maximal embedding dimension.
\end{abstract}

\section{Introduction}

A \vocab{numerical semigroup} $\Lambda$ is a subset of the nonnegative integers $\mathbb{N}_0$ that has finite complement and is also an additive monoid, that is,~it contains 0 and is closed under addition. Numerical semigroups were first studied by Frobenius and Sylvester in the context of the Frobenius coin problem, and more recently have been of interest in constructions in algebraic geometry. See~\cite{kaplan_2017, rosales_2009} for more background. Given a numerical semigroup $\Lambda$, one can define its \vocab{genus} $g(\Lambda) := \#(\mathbb{N}_0 \setminus \Lambda)$; its \vocab{multiplicity} $m(\Lambda) := \min (\Lambda \setminus \{0\})$; its \vocab{conductor} $c(\Lambda) := \min \{ c \in \mathbb{N}_0 : c + \mathbb{N}_0 \subseteq \Lambda \}$; its \vocab{Frobenius number} $f(\Lambda) := c(\Lambda) - 1$; and its \vocab{depth} $q(\Lambda) = \ceil{c(\Lambda)/m(\Lambda)}$, the last of which was recently introduced by Eliahou and Fromentin~\cite{ef_2020}.

A central problem in the study of numerical semigroups is enumeration after ordering by a specific invariant. For instance, in 2011, Zhai~\cite{zhai_2012} proved that the number of numerical semigroups with genus $g$ is asymptotically $C\left( \frac{1+\sqrt5}{2} \right)^g$, where $C$ is a positive constant, settling two conjectures of Bras-Amor\'os~\cite{bras-amoros_2007}. In 1990, Backelin~\cite{backelin_1990} estimated the number $\fr(f)$ of numerical semigroups with Frobenius number $f$, answering a question of Wilf:

\begin{theorem}[{Backelin \cite[Prop.~1]{backelin_1990}}]
  There are constants $C_0$ and $C_1$ for which
  \[\fr(f) \sim \begin{cases}
    C_0 \cdot 2^{f/2} & \text{if }f \text{ even}, \\
    C_1 \cdot 2^{f/2} & \text{if }f \text{ odd}.
  \end{cases}\]

  \label{thm:back0}
\end{theorem}

A key ingredient in both of these proofs is that, after ordering, almost all numerical semigroups have small depth. Zhai shows that numerical semigroups with genus $g$ are almost all of depth 2 or 3; Backelin shows that numerical semigroups with Frobenius number $f$ are almost all of depth 2, 3, or 4. They then estimate the number of numerical semigroups with these depths. For instance, letting $\fr_q(f)$ denote the number of numerical semigroups with Frobenius number $f$ and depth $q$, Backelin shows the following enumerative result.

\begin{proposition}[{Backelin \cite[pg.~202]{backelin_1990}}]
  We have $\fr_2(f) + \fr_3(f) + \fr_4(f) \leq 3 \cdot 2^{\floor{(f-1)/2}}$.
\end{proposition}

It is natural to consider the number of numerical semigroups of larger depth, especially since such bounds would sharpen the results of Zhai and Backelin to lower-order terms. For instance, in 2023, Zhu~\cite{zhu_2022} proves the sharpest known asymptotics of the number of numerical semigroups by bounding the number of semigroups of depth at least $4$, refining estimates of Bacher~\cite{bacher_2019}.

In the realm of Frobenius numbers, Backelin~\cite[eq.~13, eq.~15]{backelin_1990} proves~\Cref{thm:back0} by bounding on the number of numerical semigroups with depths $3$, $4$, and at least $5$. In particular, Backelin shows that the number of depth-3 and depth-4 numerical semigroups with Frobenius number $f$ and multiplicity $m$ is at most $\frac12 (11/6)^{\floor{(f-3)/2}} (12/11)^m$. The results of Backelin have since remained the best bounds on the number of numerical semigroups of large depth, but unfortunately they are not sharp. The bottleneck seems to be that numerical semigroups of large depth are significantly more challenging to characterize than semigroups of depth 2 and 3.

To aid with computation, Blanco and Rosales~\cite{blanco_rosales_2012} develop algorithms to compute the number of numerical semigroups with fixed Frobenius number, and recently Branco, Ojeda, and Rosales~\cite{bor_2021} developed algorithms to compute the number of numerical semigroups with fixed multiplicity and Frobenius number. However, enumerations for small $f$ and $m$ say little about the behavior as these invariants grow large.

In this paper, we determine the number of numerical semigroups with fixed Frobenius number, multiplicity, and depth. Our results are sharp up to a subexponential factor, thus improving Backelin's bounds by an exponential factor for all $q \geq 3$ and yielding an alternative proof to \Cref{thm:back0}. We abbreviate the constant $c_q := \sqrt{\floor{(q+2)^2/4}}$ for $q \geq 1$.

\begin{theorem}
  \label{thm:frob-count}
  Fix $q \geq 3$, and let $(f,m)$ vary subject to the condition $(q-1)m < f < qm$. Then the number of numerical semigroups with multiplicity $m$ and Frobenius number $f$ is $\left(c_q^{\alpha} c_{q-1}^{1-\alpha}\right)^{m+o(m)}$, where $\alpha := f/m - (q-1)$.
\end{theorem}

As a corollary, we obtain sharp asymptotics on the number $\fr_q(f)$ of depth-$q$ numerical semigroups with Frobenius number $f$, which generalizes Backelin's count of semigroups of depth $2$, $3$, and $4$ to arbitrary depth. We also obtain sharp asymptotics on the number of depth-$q$ numerical semigroups with multiplicity $m$, which we denote by $\mult_q(m)$. 

\begin{corollary}
  \label{cor:q-depth}
  We have $\fr_q(f) = \left(c_{q-1}\right)^{f/(q-1) + o(f)}$ for $q \geq 2$.
\end{corollary}

\begin{corollary}
  \label{cor:q-depth-mult}
  We have $\mult_q(m) = c_q^{m + o(m)}$ for $q \geq 2$.
\end{corollary}

\Cref{thm:frob-count} also allows us to determine the asymptotic number of semigroups with maximal embedding dimension, settling a recent conjecture of Singhal~\cite{singhal_2022}. Let $\MED(f)$ denote the number of numerical semigroups with maximal embedding dimension and Frobenius number $f$.

\begin{corollary}
  \label{thm:med-count}
  We have $\MED(f) = 2^{f/3 + o(f)}$.
\end{corollary}

Our methods to count semigroups of large depth differ from the literature (for example,~from those in~\cite{zhai_2012}) in that they are very additive-combinatorial in nature. To bound the number of semigroups of fixed genus, Zhu~\cite{zhu_2022} uses Kunz words and reduces the problem to a graph homomorphism inequality due to Zhao~\cite{zhao_2011}. We extensively use and refine these methods to prove bounds in the case of fixed Frobenius number. Many of our results from these methods can be phrased in terms of coloring problems on Schur triples. This furthers the connection between numerical semigroups, graph homomorphisms, and additive combinatorics. 

Finally, we strengthen existing results on the distribution of multiplicity and genus with respect to Frobenius number, akin to Zhu's results for the distribution of multiplicity with respect to genus~\cite[\S 6]{zhu_2022}. 
Backelin~\cite{backelin_1990} shows that for a fixed Frobenius number $f$, almost all semigroups have multiplicity $f/2+o(f)$; Singhal~\cite{singhal_2022} shows that almost all semigroups with Frobenius number $f$ have genus $3f/4 + o(f)$. We strengthen these results by computing the exact limiting distributions in~\Cref{thm:mult-dist} and~\Cref{thm:gen-dist}. 
These results differ from the first half of the paper in that they arise from an analysis of semigroups of depths $2$ and $3$ instead of arbitrary depth, though some of the results depend on \Cref{thm:frob-count} for $q=3$. We also show that~\Cref{thm:frob-count} implies that almost all depth-$q$ numerical semigroups with Frobenius number $f$ have multiplicity close to $f/(q-1)$ for arbitrary depth $q \geq 3$, though we do not determine the exact limiting distributions in this case.

\subsection*{Outline} In~\Cref{sec:defs}, we introduce Kunz words and discuss the properties of the constants $c_q$. \Cref{sec:fixed-depth} establishes the key enumerative results on numerical semigroups of fixed depth which form the backbone of our paper. We apply these results in~\Cref{sec:count} to count numerical semigroups and semigroups with maximal embedding dimension. In~\Cref{sec:stats}, we prove some statistics on the distribution of multiplicity and genus. Finally, we discuss possible directions of research---sharpening asymptotics, enumerating other kinds of numerical semigroups, and studying general polychromatic Schur problems---in~\Cref{sec:direcs}.

\section*{Acknowledgements} 

This work was done at the University of Minnesota Duluth with support from Jane Street Capital, the National Security Agency (grant number H98230-22-1-0015), and the Izzo Fund. The author especially thanks Joseph Gallian for his mentorship and for nurturing a wonderful environment for research. We are also very grateful to Amanda Burcroff and Brian Lawrence for detailed revisions; to Colin Defant, Nathan Kaplan, Noah Kravitz, Deepesh Singhal, and Daniel Zhu for helpful discussions; and to the anonymous referees for their helpful comments. 

\section{Kunz words and the constants $c_q$}
\label{sec:defs}

In this section, we define Kunz words, prove some initial bounds on the number of Kunz words, and discuss some properties of the sequence $(c_q)_{q \geq 1}$.

\subsection{Kunz words}

Following Zhu~\cite[\S 3]{zhu_2022}, we define the Kunz word of a numerical semigroup.

\begin{definition}
  Let $\Lambda$ be a numerical semigroup with multiplicity $m$. The \vocab{Kunz word} $\cK(\Lambda)$ of $\Lambda$ consists of $m-1$ integers $w_1 \cdots w_{m-1}$ where $m \cdot w_i + i$ is the least element of $\Lambda$ that is equivalent to $i \pmod{m}$.
\end{definition}

All of our examples will have $w_i$ a digit from $1$ to $9$ for clarity, but the entries can be arbitrarily large in the general case.

\begin{example}
  If $\Lambda = \NN_0 \setminus \{1,2,3,4,5,7,9,10,13\}$, then $\cK(\Lambda) = 31221$.
\end{example}

Note that $\Lambda$ is generated by $\{m, m\cdot w_1 + 1, \dots, m \cdot w_{m-1} + (m-1)\}$ (known as the \vocab{Ap\'ery set} of $\Lambda$), so $\cK$ is a bijection between numerical semigroups of multiplicity $m$ and Kunz words of length $m-1$.

The quantity $\cK(\Lambda)$ is more traditionally regarded as a vector of scalars $(w_1, \dots, w_{m-1})$, known as the \vocab{Kunz coordinate vector} of $\Lambda$. Introduced by Kunz in 1987, the Kunz coordinate vector is typically regarded as a point in $(m-1)$-dimensional space; authors have used lattice methods, such as Ehrhart theory~\cite[Thm.~9]{kaplan_2011}, to count numerical semigroups. Recently, Bacher~\cite{bacher_2019} and Zhu~\cite{zhu_2022} regard $\cK(\Lambda)$ as a word whose entries are constrained by the Kunz conditions, then enumerate these combinatorially. We adopt the latter approach for this paper, and hence we opt to speak in terms of Kunz words (as opposed to coordinates) to emphasize our approach within our notation.

Since $\cK(\Lambda)$ determines $\Lambda$, we can read of the invariants of $\Lambda$ from $\cK(\Lambda)$ (for a proof of this result, see, for instance,~\cite[Prop.~3.4]{zhu_2022}). Let $[\ell] = \{1,\dots, \ell\}$.

\begin{proposition}[{Kunz~\cite[$\S2$]{kunz_1987}}]
  \label{prop:read-read-read}
  Let $\Lambda$ be a numerical semigroup and $\cK(\Lambda) = w_1 \cdots w_\ell$. Then
  \begin{itemize}
    \item $m = \ell+1$;
    \item $g = \sum_{i \in [\ell]} w_i$;
    \item $q = \max_{i \in [\ell]} w_i$;
    \item $f = (\ell+1)(q-1) + j$, where $j$ is maximal such that $w_j = q$.
  \end{itemize}
\end{proposition}

We often refer to invariants of a numerical semigroup via its Kunz word; for instance, the genus of the Kunz word $31221$ is $9$. However, we never refer to the multiplicity of a Kunz word and instead speak of its length, which we denote $\ell := m-1$ for brevity.

\begin{definition}
  A Kunz word is \vocab{$q$-Kunz} if its depth (that is, its maximum entry) is at most $q$.
\end{definition}

It is natural to ask which words correspond to valid Kunz words of a numerical semigroup. This question was resolved by Kunz:

\begin{proposition}[Kunz~{\cite{kunz_1987,rosales_2002}}]
  A word $w_1 \cdots w_{\ell}$ with positive integer entries is a valid Kunz word if:
  \begin{itemize}
    \item $w_i + w_j \geq w_{i+j}$ for all $i,j$ with $i+j \leq \ell$; and
    \item $w_i + w_j + 1 \geq w_{i+j-\ell-1}$ for all $i,j$ with $i+j > \ell + 1$.
  \end{itemize}
\end{proposition}

These inequalities are the \vocab{Kunz conditions}.

\begin{definition}
For a fixed length $\ell$, we say the word $w_1 \cdots w_{\ell}$ is \vocab{Kunz} if its entries satisfy the Kunz conditions.
\end{definition}

It will be useful to work exclusively with sets of Kunz words with fixed invariants. Hence, we let $\cK(f,\ell)$ denote the set of Kunz words with Frobenius number $f$ and length $\ell$. When a parameter is set to $\NN$, we let that parameter vary; for instance, $\cK(\NN,\ell)$ denotes the set of Kunz words with length $\ell$. In lieu of $\cK$, we use $\cK_q$ to denote $q$-Kunz words instead of general Kunz words. 

A word with depth $q$ and length $\ell$ has Frobenius number at most $\ah(q,\ell) := (\ell+1)q - 1$, so we occasionally write $\cK_q(\ah, \ell) := \cK_q(\ah(q,\ell), \ell)$ for the set of $q$-Kunz words with maximum Frobenius number. These correspond to words whose last entry is $q$.\footnote{For $q=3$, Zhu~\cite[Def.~3.7]{zhu_2022} calls these \emph{stressed} $q$-Kunz words, hence the notational choice of $\ah$.}

\begin{example}
  We have 
  \begin{itemize}
    \item $\cK(5,\NN) = \{11111, 12, 211, 22, 3\}$;
    \item $\cK_3(\NN,2) = \{11,12,21,22,23,31,32,33\}$; 
    \item $\cK(5,2) = \cK_3(5,2) = \{12, 22\}$.
    \item $\cK_3(\ah;2) = \{23, 33\}$.
  \end{itemize}
\end{example}

Note that we have $\fr(f) = \# \cK(f,\NN)$ and $\fr_q(f) = \# \cK_q(f,\NN)$, as well as $\mult_q(m) = \#\cK_q(\NN,m-1)$. (Here, we let $\#\mathcal{S}$ be the cardinality of the set $\mathcal{S}$.) Thus, we first focus on bounding the size of families of Kunz words like $\#\cK_q(f,\ell)$ in \Cref{sec:fixed-depth}, which in turn helps us bound the number of numerical semigroups with fixed invariants in \Cref{sec:count}. 

\subsection{Initial bounds on $\#\cK_q$}

We prove the following initial bounds on $\# \cK_q(\NN, \ell)$ and $\# \cK_q( f, \NN)$ which, while much weaker than our main results, helps with some technicalities in~\Cref{sec:count}. 

\begin{proposition}
  \label{prop:stupid-ell}
  We have $\# \cK_q(\NN,\ell) \leq q^{\ell}$.
\end{proposition}

\begin{proof}
  This is immediate from the definition since $\cK_q(\NN,\ell) \subseteq [q]^\ell$.
\end{proof}

\begin{corollary}
  \label{cor:stupid-f}
  We have $\# \cK_q(f,\NN) \leq f\cdot q^{f/(q-1)}$. 
\end{corollary}

\begin{proof}
  A Kunz word with Frobenius number $f$ and length $\ell$ has depth $q$ if and only if we have $q = \ceil{(f+1)/(\ell+1)}$, that is,~when $\ell \in L_{f,q} := \left[ \frac{f+1-q}{q}, \frac{f+1-q}{q-1}\right) \cap \NN$. Hence, we can use~\Cref{prop:stupid-ell} to bound:
  \[
  \begin{aligned}
    \# \cK_q(f,\NN) &= \sum_{\ell \in L_{f,q}} \# \cK_q(f,\ell) 
    \leq \sum_{\ell \in L_{f,q}} \# \cK_q(\NN,\ell) \\
    &\leq \sum_{\ell \in L_{f,q}} q^{\ell}
    < \sum_{\ell \in L_{f,q}} q^{(f+1-q)/(q-1)} \\
    &= \frac1q \left| L_{f,q} \right| q^{f/(q-1)} 
  < f \cdot q^{f/(q-1)},
\end{aligned}
\]
  as desired.
\end{proof}

We also exactly count Kunz words of small depth, which mirrors Backelin's enumeration of numerical semigroups of small depth. This will later be useful to reprove~\Cref{thm:back0} and to determine distributions on numerical semigroups in~\Cref{sec:stats}.

\begin{proposition}
  \label{prop:2-3-count}
  We have the following exact enumerative results:
  \begin{itemize}
    \item[(i)] $\# \cK_2(f,\ell) = 2^{f-2-\ell}$ when $(f-1)/2 \leq \ell \leq f-2$.
    \item[(ii)] $\# \cK_3(f,\ell) = 2^{\ell-j} \cdot \# \cK_3(\ah,j)$ when $(f-2)/3 \leq \ell \leq (f-3)/2$, where $j = f-2-2\ell$.
  \end{itemize}
  For other $\ell$, these quantities are $0$.
\end{proposition}

\begin{proof}
  Part (i) is not difficult and is shown in, for instance,~\cite[Thm.~10]{singhal_2022}, since all words with 1's and 2's are 2-Kunz. For part (ii), a $3$-Kunz word $W = w_1 \cdots w_\ell$ has $2\ell + j = f-2$ by~\Cref{prop:read-read-read}, which means $\ell \in [(f-2)/3, (f-3)/2]$. For $W$ to be 3-Kunz, it must satisfy the following conditions: 
  \begin{itemize}
    \item $w_{j+1} \cdots w_{\ell} \in [2]^{\ell-j}$, since otherwise $j$ is not maximal for which $w_j = 3$; and
    \item $w_1 \cdots w_j \in \cK_3(\ah, j)$, since $w_1\dots w_j$ is itself $3$-Kunz and $w_j = 3$.
  \end{itemize}
  These conditions are also sufficient, so $\# \cK_3(f,\ell) = 2^{\ell-j} \cdot \# \cK_3(\ah,j)$. The result follows.
\end{proof}

\begin{remark}
  \label{rem:thisishard}
  It is tempting to directly estimate the size of $\cK_3(\ah,j)$ using the Kunz condition arising from $w_j = 3$. Namely, since $w_1 \cdots w_{j} \in \cK_3(\ah,j)$ forces $w_j = 3$, we have that $w_i + w_{j-1-i} \geq 3$ for $i=1$, $2$, \dots, $\floor{(j-1)/2}$. Hence, $w_i$ and $w_{j-1-i}$ cannot both be $1$'s, which implies that
  \begin{itemize}
    \item $\#\cK_3(\ah,j) \leq 8^{(j-1)/2}$ when $j$ is odd, and
    \item $\#\cK_3(\ah,j) \leq 2 \cdot 8^{(j-2)/2}$ when $j$ is even.
  \end{itemize}
  However, these naive bounds are not sharp, since not all words with $w_i + w_{j-1-i} \geq 3$ are necessarily Kunz. In fact, \Cref{thm:frob-count} implies that $\#\cK_3(\ah,j) = 6^{j/2 + o(j)}$, so the naive bounds are an exponential factor away from the truth. Even worse, one can check that the naive bounds give a bound of roughly $\#\cK_3(f, \NN) \leq \frac{1}{6}f 2^{f/2}$, which is not even enough to reprove~\Cref{thm:back0}.

  The naivet\'e is similarly bad for large depth: there are $\frac12 (q^2 + 3q - 2)$ solutions to $a + b \geq q$ where $a,b \in [q]$, so the naive method gives a bound of $\#\cK_q(\ah, j) \leq \left( \frac12 (q^2 + 3q - 2) \right)^{(j-1)/2}$, which is off from the true answer given by~\Cref{thm:frob-count} by a factor of about $2^{j/2}$. This is essentially Backelin's bound for $q \geq 5$; see~\cite[eq.~15]{backelin_1990}. Thus, we must be more careful in bounding quantities like $\#\cK_q(\ah, j)$ for $q \geq 3$, which is essentially the focus of~\Cref{ss:tail-heavy}.
\end{remark}

\subsection{The constants $c_q$}

Throughout the paper, we abbreviate the constant $c_q := \sqrt{\floor{(q+2)^2/4}}$. For even $q$, we have $c_q = (q+2)/2$; for odd $q$ we have $c_q = \frac12 \sqrt{(q+1)(q+3)}$. Evidently, $(c_q)_{q \geq 1}$ is an increasing sequence. The following inequalities will be helpful for bounding purposes.

\begin{lemma}
  \label{lem:annoying}
  For any $r \in [0,1]$, the sequence $(c_q^{1/(q+r)})_{q \geq 2}$ is decreasing.
\end{lemma}

\begin{proof}
  We wish to show $c_q^{q+r+1} > c_{q+1}^{q+r}$, or $c_q > \left( c_{q+1}/c_q \right)^{q+r}$. We have that
  \[c_{q+1}/c_q = \begin{cases}
    \sqrt{\frac{q+4}{q+2}} & \text{if }q\text{ even}, \\[8pt]
    \sqrt{\frac{q+3}{q+1}} & \text{if }q\text{ odd}.
  \end{cases}\]
  Hence, $c_{q+1}/c_q \leq \sqrt{\frac{q+3}{q+1}}$ in all cases. Therefore,
  \[\left( \frac{c_{q+1}}{c_q} \right)^{q+r} \leq \left( 1 + \frac{2}{q+1} \right)^{\frac{q+r}{2}} < e,\]
  since $r \in [0,1]$. This shows the result for $q \geq 4$. The result can be checked manually for $q=2$, $3$.

\end{proof}

\begin{corollary}
  \label{cor:cq-dec}
  For fixed $q \geq 3$ and $r \in [0,1]$, the quantity $F(t) = (c_q^t c_{q-1}^{1-t})^{1/(q+t-r)}$ is decreasing on $t \in [0,1]$.
\end{corollary}

\begin{proof}
  Note that $\ln F(t) = \frac{t}{q+t-r} \ln c_q + \frac{1-t}{q+t-r} \ln c_{q-1}$, which has derivative 
  \[\frac{F'(t)}{F(t)} = \frac{1}{(q+t-r)^2} \left( (q-r) \ln c_q - (q+1-r) \ln c_{q-1} \right).\]
  By~\Cref{lem:annoying}, this derivative is always negative, which implies $F'(t) < 0$ since $F(t) > 0$.
\end{proof}

\section{Kunz words of fixed depth}
\label{sec:fixed-depth}

In this section, we prove~\Cref{thm:frob-count}, which forms the backbone of our paper by implying a large number of enumerative results, namely Corollaries \ref{cor:q-depth}, \ref{cor:q-depth-mult}, and \ref{thm:med-count}. Moreover, the case of $q=3$ plays an important role in showing the convergence of certain sums in the proofs of~\Cref{cor:mult-const} and \Cref{thm:genus-const}. 

The central idea is that every instance of $q$ in a $q$-Kunz words imposes a strong Kunz condition on the other entries. Thus, a cluster of $q$'s imposes a large number of restrictions and limits the number of valid Kunz words, so much so that we can actually disregard the other conditions and still show the necessary upper bound. Then, we will represent each of these restrictions with an edge to form an almost-regular graph on the entries of the word.

\subsection{Graph lemmas} To bound the number of $q$-Kunz words, we bound the number of colorings of certain graphs arising from these Kunz words. In this section, we establish the graph lemmas needed for our proofs. Of great use to us is the following graph homomorphism lemma due to Zhao~\cite{zhao_2011}.

\begin{definition}
  A graph $H$, possibly with loops, is a \vocab{threshold graph} if there exists a labeling $g: V(H) \to \mathbb{R}$ and threshold $\lambda$ such that for (possibly equal) $u,v \in V(H)$, we have $uv \in E(H)$ if and only if $g(u) + g(v) \geq \lambda$.

  The graph $H_{q}$ is the threshold graph with $q$ vertices labeled $1$, $2$, \dots, $q$ and threshold $\lambda = q$.
\end{definition}

\begin{theorem}[{Zhao \cite[$\S2$, pg.~663]{zhao_2011}\footnote{Strangely, the result is not explicitly written within the paper, but rather follows from a series of implications at the bottom of pg.~663 in~\cite{zhao_2011}; Zhao proves that all threshold graphs are GT graphs, that is, they satisfy the given inequality.}}]

  \label{thm:zhao-hom}
  If $G$ is a loop-free, $d$-regular graph and $H$ is a thres\-hold graph, then $\operatorname{hom}(G,H) \leq \operatorname{hom}(K_{d,d},H)^{\# V(G)/(2d)}$.

  (Here, $\hom(G,H)$ is the number of graph homomorphisms from $G$ to $H$. Recall a \emph{graph homomorphism} is a function $\varphi : V(G) \to V(H)$ such that for any $x, y \in V(G)$, we have $xy \in E(G)$ only if $\varphi(x)\varphi(y) \in E(H)$.)
\end{theorem}

The hypothesis of~\Cref{thm:zhao-hom} requires that $G$ is $d$-regular, but many of the graphs we work with are ``almost regular'' in the sense that almost, but not all, vertices have degree $d$. We perturb these graphs to be $d$-regular without decreasing the number of graph homomorphisms using the following technical lemma.

\begin{definition}
  For a graph $G$, let $D_d(G) := d \cdot \# V(G) - 2 \cdot \# E(G)$. The quantity $D_d(G)$ can be viewed as the ``discrepancy'' $\sum_{v \in V(G)} (d - \deg v)$ between $G$ and an $d$-regular graph.
\end{definition}

\begin{lemma}
  \label{lem:perturb}
 Let $G$ be any graph with maximum degree at most $d$, and abbreviate $D := D_d(G)$. There is a $d$-regular graph $G'$ with the following properties:
  \begin{itemize}
    \item $\# V(G') \leq 1 + \max\{3 + D/d, 2\ceil{d/2}\} + \#V(G)$; and
    \item $ \hom(G,H_q) \leq \hom(G',H_q)$. 
  \end{itemize}
\end{lemma}

\begin{proof}
  Before modifying $G$, color every vertex in $V(G)$ blue. For a two-colored graph $G$, say a homomorphism from $G$ to $H_q$ is \vocab{admissible} if every red vertex is mapped to the $H_q$-vertex labeled $q$. Then the following operations do not decrease the number of admissible homomorphisms: 
  \begin{itemize}
    \item removing an edge between blue vertices;
    \item adding isolated vertices (either red or blue);
    \item adding edges between two vertices, not both blue. 
  \end{itemize}
  The first operation decreases the number of restrictions on blue vertices, and for the last two operations, we can send every red vertex with the $H_q$-vertex labeled $q$ without adding restrictions to the blue vertices.
  Hence, we perform the following procedure on $G$ to get a new graph $G'$:
  \begin{itemize}
    \item[(1)] If $\# V(G)$ is odd, add an isolated, unlabeled blue vertex (so now $\# V(G)$ is even).
    \item[(2)] Remove edges between blue vertices until $D := D_d(G)$ is divisible by $d$. 
    \item[(3)] If $D/d > d$, add $D/d$ isolated red vertices. Otherwise, add $2\ceil {d/2}$ red vertices. 
    \item[(4)] For each blue vertex $v_b$ with degree $x < d$, draw $d-x$ edges from $v_b$ to the $d-x$ red vertices with least degree (breaking ties arbitrarily).
    \item[(5)] Add an edge between two red vertices of least degree (breaking ties arbitrarily) until all red vertices have degree $d$.
  \end{itemize}
  We claim that every step is valid and that the procedure finally yields a $d$-regular graph $G'$. Step (1) causes no issues and increases $D$ by $d$ if we add a new vertex. Step (2) is possible since removing an edge increases $D$ by $2$, and $\# V(G)$ is even so $D$ is even; we remove at most $d$ edges, so $D$ increases by at most $2d$. Step (3) is also fine; we add at most $\max \{ D/d, 2\ceil{d/2}\}$ edges at this point. 

  Step (4) allows all blue vertices to have degree $d$; since there are at least $D/d$ red vertices, no red vertex will have degree more than $d$, and since there at least $d$ red vertices, the $d-x$ red vertices we choose are all distinct. Step (5) is only done if $D/d < d$, in which case we have an even number of red vertices whose degrees total to $D$ and pairwise differ by at most $1$. Adding an edge maintains the latter property, while increasing the total degree by $2$. After adding $d\ceil{d/2} - D/2$ edges, the total degree is $d\cdot 2\ceil{d/2}$ and so every red vertex will have degree $d$.

  Voil\`a! The end graph $G'$ is $d$-regular and has at most $1 + \max\{3+D/d,2\ceil{d/2}\} + \#V(G)$ vertices. It has at least as many admissible homomorphisms as $G$, the latter of which has exactly $\hom(G,H_q)$ admissible homomorphisms. Every admissible homomorphism is a homomorphism, so we have $\hom(G',H_q) \geq \hom(G,H_q)$ as desired.
\end{proof}

We also use the following estimate on graph homomorphisms in our calculations.

\begin{lemma}
  \label{lem:hom-cnt}
  We have that $\hom(K_{d,d}, H_{q}) \leq 2q \cdot c_q^{2d}$.
\end{lemma}

\begin{proof}
  Suppose some vertex $v$ in $K_{d,d}$ is sent to the $H_{q}$-vertex with label $m$, where $m$ is minimal. There are two ways to choose the bipartite half of $v$, and vertices in this half must be sent to $H_{q}$-vertices with labels in $[m, q]$. Then the other half must be sent to $H_{q}$ vertices with label at least $q-m$. Hence, for fixed $m$, there are at most $2 (q-m+1)^d (m+1)^d$ homomorphisms. So the total number of homomorphisms is at most
  \[\sum_{m=1}^q 2(q-m+1)^d (m+1)^d \leq \sum_{m=1}^q 2 \left\lfloor \frac{(q+2)^2}{4} \right\rfloor^d = 2q \cdot c_q^{2d},\]
  since $(q-m+1)(m+1)$ is at most $(q+2)^2/4$ by the AM-GM inequality and is an integer.
\end{proof}

\subsection{$t$-tail-heavy words}
\label{ss:tail-heavy}

In the spirit of using clusters of $q$'s to impose Kunz conditions, we make the following definition:

\begin{definition}
  A (not necessarily Kunz) word $w_1 \cdots w_{\ell} \in [q]^\ell$ is \vocab{$t$-tail-heavy} of depth $q$ if:
  \begin{itemize}
    \item there exist $n > \sqrt{\ell}$ indices\footnote{Here, the choice of $\sqrt{\ell}$ is not important; we may use any sub-$\left( \frac{\ell}{\log \ell} \right)$ function which goes to infinity as $\ell$ grows large.} $i_1, \dots, i_n$ greater than $\ell - t$ such that $w_{i_1} = \dots = w_{i_n} = q$;
    \item for any $x,y \leq \ell - t$ such that $x+y \in \{i_1, \dots, i_n\}$, we have $w_x + w_y \geq q$.
  \end{itemize}
\end{definition}

In essence, $t$-tail-heavy words of depth $q$ are words with a large number of $q$'s amongst the last $t$ entries, restricted by the Kunz conditions arising from these $q$'s. We prove the following estimate on the number of $t$-tail-heavy words of depth $q$.

\begin{theorem}
  \label{thm:tail-heavy}
  The number of $t$-tail-heavy words of length $\ell$ and depth $q$ is at most $tq^{t} c_q^{\ell + \sqrt{\ell}+10}$.
\end{theorem}

\begin{proof}[Proof]
  Abbreviate $h := \ell - t$, and suppose $W = w_1 \cdots w_{\ell}$ is $t$-tail-heavy of depth $q$. There are at most $q^{t}$ ways to choose the last $t$ entries $w_{h+1}, \dots, w_{\ell}$, so it suffices to count the number of ways to choose $W_H := w_1 \cdots w_h$. Let $w_{i_1} = w_{i_2} = \dots = w_{i_n} = q$ with $n > \sqrt{\ell}$. Any choice of $W_H$ must have $w_x + w_y \geq q$ if $x+y \in \{i_1, \dots, i_n\}$. 

  Construct a graph $G$ with $h$ vertices labeled $1,2,\dots, h$ such that $xy \in E(G)$ if $x+y \in \{i_1, \dots, i_n\}$. By the above discussion, the number of choices for $W_H$ is at most $\operatorname{hom}(G,H_{q})$. We cannot directly use~\Cref{thm:zhao-hom} since $G$ is not regular, but $G$ has enough vertices of degree $n$ that we will perturb $G$ to make an $n$-regular graph. Namely, every vertex with labels $t+1, \dots, h$ has degree $n$ other than those with label in $\{i_1/2, \dots, i_n/2\}$, which have degree $n-1$. Hence, we have that $D_n(G) \leq n + nt$, so by~\Cref{lem:perturb} there is an $n$-regular graph $G'$ with 
  \[\#V(G') \leq 1 +\max\{3+D/n, 2\ceil{n/2}\} + \# V(G) = 1 + (4+t) + h = \ell + 5\] and $\hom(G',H_q) \geq \hom(G, H_q)$. Now, we have that
  \[\hom(G,H_q) \leq \hom(G',H_{q}) \leq \hom(K_{n,n}, H_{q})^{(\ell+5)/(2n)} \leq (2q \cdot c_q^{2n})^{(\ell+5)/(2n)}\]
  by~\Cref{lem:hom-cnt}. Hence, the number of $t$-tail-heavy words of depth $q$ is at most
  \begin{align*}
  q^t \sum_{\sqrt{\ell} < n \leq t} (2q \cdot c_q^{2n})^{(\ell+5)/(2n)} &\leq q^t \sum_{\sqrt{\ell} < n \leq t} (c_q^{2n+2})^{(\ell+5)/(2n)} \\
  & \leq q^t \sum_{\sqrt{\ell} < n \leq t} c_q^{(\ell+5)(1+1/\sqrt{\ell})} \\
  & \leq q^t \sum_{\sqrt{\ell} < n \leq t} c_q^{\ell + \sqrt{\ell}+10} \\
  & \leq tq^t c_q^{\ell + \sqrt{\ell}+10},
\end{align*}
as we had sought.\end{proof}

\subsection{$q$-Kunz words}

We are almost ready to prove~\Cref{thm:frob-count}. We only need a short lemma about an operation sending $q$-Kunz words to $(q-1)$-Kunz words.

\begin{lemma}
  \label{lem:decrease}
  If $w_1 \cdots w_\ell$ is a $q$-Kunz word and $v_i = \min \{ q-1, w_i \}$, then $v_1\cdots v_{\ell}$ is $(q-1)$-Kunz.
\end{lemma}

\begin{proof}
  For any pair $(i,j)$ with $i+j \le \ell$, consider the quantities $v_i + v_j$ and $v_{i+j}$. If $q-1 \in \{v_i, v_j\}$, then we have $v_i + v_j > q-1 \geq v_{i+j}$. If not, then we have $(v_i, v_j) = (w_i, w_j)$ and it follows that $v_i + v_j = w_i + w_j \geq w_{i+j} \geq v_{i+j}$. A similar analysis is true for the case $i+j > \ell + 1$.
\end{proof}

Every word consisting of 1's and 2's is 2-Kunz, so it is not hard to see that for $q=2$ the number of numerical semigroups with multiplicity $m$ and Frobenius number $f \in (m, 2m)$ is $2^{\alpha m + O(1)}$, where $\alpha = (f-m)/m$. This will serve as our base case for the induction in our proof; although strictly speaking~\Cref{thm:frob-count} is not true for $q=2$, the only important part of the inductive step is that, asymptotically, there are fewer words of depth $q-1$ than depth $q$.

\begin{proof}[Proof of~\Cref{thm:frob-count}]

  We prove the result by induction, with the ``base case'' of $q=2$ (see the above discussion). Throughout, we work directly with Kunz words of length $\ell = m-1$, and we wish to bound the size of $\#\cK_q(f,\ell)$ from below and from above.
  
  \paragraph{Lower bound.} Let $j = f - m(q-1) = \alpha m$. Consider the family of words $w_1 \cdots w_{\ell}$ governed by the following conditions: we must have $w_j = q$ and require that $w_i$ is in the interval
  \begin{itemize}
    \item $[\floor{(q+1)/2}, q]$ if $i \leq j/2$;
    \item $[\floor{q/2}, q]$ if $j/2 < i < j$;
    \item $[\floor{q/2}, q-1]$ if $j < i \leq (\ell+j+1)/2$; and
    \item $[\floor{(q-1)/2}, q-1]$ if $i > (\ell+j+1)/2$.
  \end{itemize}

  Recall from~\Cref{prop:read-read-read} that the Frobenius number of the words in this family is equal to $(\ell+1)(q-1) + j = f$. We claim these words are $q$-Kunz. Indeed, the Kunz conditions are guaranteed if we have $w_x + w_y \geq q$ or $w_x + w_y = q-1$ and $x+y > j$. The only other case is when $x,y > (\ell+j+1)/2$, in which case $x+y - \ell - 1 > j$ and so $w_x + w_y + 1 \geq q-1 \geq w_{x+y-\ell-1}$, as desired.

  There are precisely
  \[\left\lfloor \frac{(q+2)}{2} \right\rfloor^{\floor{j/2}} \left \lfloor \frac{(q+3)}{2} \right\rfloor^{\floor{(j-1)/2}} \cdot \left \lfloor \frac{(q+1)}{2} \right\rfloor^{\floor{(\ell-j+1)/2}} \left\lfloor \frac{(q+2)}{2} \right\rfloor^{\floor{(\ell-j)/2}}\]
  words in the constructed family. Since $j = \alpha m$, we have that $\lfloor j/2 \rfloor$ and $\lfloor (j-1)/2 \rfloor$ are both $\alpha m/2 + O(1)$, and $\lfloor (\ell -j+1)/2\rfloor$ and $\lfloor (\ell-j)/2 \rfloor$ are both $(1-\alpha)m/2 + O(1)$. Moreover, we have
  \[c_q^2 = \left\lfloor \frac{(q+2)^2}{4} \right\rfloor = \left\lfloor \frac{(q+2)}{2} \right\rfloor \left\lfloor \frac{(q+3)}2 \right\rfloor. \]
  Hence, the left factor is $c_q^{\alpha m + O(1)}$, while the right factor is $c_{q-1}^{(1-\alpha)m + O(1)}$. This is enough for the lower bound.

  \paragraph{Upper bound.} Select $\eps > 0$ and let $t := \ceil{ \eps \ell}$. Split the word $W := w_1 \cdots w_{\ell}$ into chunks $W_k := w_{kt+1} w_{kt+2} \cdots w_{(k+1)t}$ for all nonnegative integers $k$ with $kt+1 \leq \ell$; every chunk has length $t$, except possibly the final chunk, and there are $C + 1 \leq \ceil{1/\eps}$ chunks in total. Say a chunk is \vocab{heavy} if it contains more than $\sqrt{\ell}$ entries equal to $q$. We split into cases depending on whether some chunk is heavy.

  \begin{quote}
    \subparagraph{Case 1: No heavy chunk.} In this case, $W$ is essentially $(q-1)$-Kunz, namely it does not have many $q$'s. Formally, each chunk has less than $\sqrt{\ell}$ entries equal to $q$, so if $t \geq 2\floor{\sqrt{\ell}}$ (which is true for $\ell$ large), then there are at most
    \[\sum_{k=0}^{\floor{\sqrt{\ell}}} \binom{t}k \leq \floor{\sqrt{\ell}} \binom{t}{\floor{\sqrt{\ell}}} \leq {\sqrt{\ell}} (\eps \ell)^{\sqrt{\ell}}\]
    ways to place the $q$'s in each chunk, for a total of at most $\left( \sqrt{\ell} (\eps \ell)^{\sqrt{\ell}} \right)^{C}$ ways to place the $q$'s overall. For each placement of $q$'s, we may use the operation in~\Cref{lem:decrease} to send $W$ to a $(q-1)$-Kunz word. Given the placement of $q$'s, the operation is injective, so for each placement we have at most $\# \cK_{q-1}(\NN , \ell)$ ways to select the remaining entries. The total number of Kunz words in this case is at most $\left( \sqrt{\ell} (\eps \ell)^{\sqrt{\ell}} \right)^{C} \cdot \# \cK_{q-1} (\NN , \ell)$.

    \subparagraph{Case 2: Some heavy chunk.} Select the largest $k$ for which $W_k$ is heavy, that is,~pick the rightmost heavy chunk. Let $\ell_k$ be the length of $W_k$ which is $t$ unless $(k+1)t > \ell$. Since $W$ has Frobenius number $f$, we must have $kt+1 < f-(\ell+1)(q-1) = \alpha (\ell+1)$. 

    The idea is to apply~\Cref{thm:tail-heavy} to the left and right sides of $W_k$; see~\Cref{fig:heavy}, where jagged segments indicate heavy chunks and arrows denote applications of the $t$-tail-heavy result. The bolded jagged region denotes the rightmost heavy chunk. 

    First, note that, by definition, the word $W_0 \cdots W_k$ is $\ell_k$-tail-heavy of depth $q$. Therefore, by~\Cref{thm:tail-heavy} we can select the first $a_k := kt+\ell_k$ entries in at most $\ell_kq^{\ell_k}c_q^{a_k + \sqrt{a_k}+10}$ ways. 

    As for the later chunks, first fix the placement of $q$'s amongst the last $C-k$ chunks and then use the operation from~\Cref{lem:decrease} on each chunk $W_{k}$, $W_{k+1}$, \dots, $W_{C}$ to obtain new chunks $V_k, V_{k+1}, \dots, V_{C}$; this map is injective given the placement of the $q$'s. Let $\operatorname{rev}(W)$ denote the reverse of the word $W$. Then note that the word
    \[V^* := \operatorname{rev}(V_{k+1} \dots V_{C}) \operatorname{rev}(V_{k})\]
    is $\ell_k$-tail-heavy of depth $q-1$ by the Kunz conditions.

    Hence, if we fix the placement of $q$'s amongst the last $b_k := \ell - a_k$ entries, there are at most $\ell_k q^{\ell_k} c_q^{b_k + \sqrt{b_k} + 10}$ ways to designate $V^*$, which in turn designates the other entries. Since the chunks to the right of $W_k$ are not heavy, there are at most $\left( \sqrt{\ell} (\eps\ell)^{\sqrt{\ell}} \right)^{C-k}$ placements of $q$'s in the last $b_k$ entries. Thus, there are at most $\left( \sqrt{\ell} (\eps\ell)^{\sqrt{\ell}} \right)^{C-k} \ell_kq^{\ell_k} c_q^{b_k + \sqrt{b_k}+10}$ ways to choose the last $b_k$ entries. The total number of Kunz words in this case is at most
    \[ \sum_{kt+1 < \alpha(\ell+1)} \left( \sqrt{\ell} (\eps \ell)^{\sqrt{\ell}} \right)^{C - k} (\ell_k)^2 q^{2\ell_k} c_q^{a_k + \sqrt{a_k}+10} c_{q-1}^{b_k + \sqrt{b_k}+10}.\]

    \begin{figure}
    \begin{tikzpicture}[scale=1.8]
\draw (0,0)--(8,0);
\foreach \i in {0,...,8} {
\draw (\i,0.05)--(\i,-0.05);
}
\node at (0,-0.05) [below] {\small $w_1$};
\node at (8,-0.05) [below] {\small $w_\ell$};
\node at (4.5,-0.1) [below] {\small $W_k$};
\draw[decorate,decoration={zigzag, amplitude=2, segment length=4}, line width = 1.5pt ] (3,0)--(4,0);
\draw[decorate,decoration={zigzag, amplitude=2, segment length=4},very thick,red, line width = 3pt] (4,0)--(5,0);
\draw[decorate,decoration={zigzag, amplitude=2, segment length=4}, line width = 1.5pt] (0,0)--(1,0);

\draw[->, line width = 1.5pt] (2,0.1) to [out=30,in=150] (4.3,0.1);
\draw[->, line width = 1.5pt] (6.5,-0.1) to [out=210,in=330] (4.7,-0.1);
\end{tikzpicture}
\caption{Proof sketch when there is at least one heavy chunk (shown jagged).}
\label{fig:heavy}
\end{figure}

  \end{quote}

\medskip

Therefore, the quantity $\cK_q(\NN , \ell)$ is bounded above by the sum of the bounds from our two cases. We now determine the asymptotics of our upper bound. The term $\left( \sqrt{\ell} (\eps \ell)^{\sqrt{\ell}} \right)^{C}$ is subexponential and, in particular, is $c_q^{o(\ell)}$. By the inductive hypothesis, $\# \cK_{q-1}( \mathbb{N} , \ell )$ is asymptotically less than $c_q^{\ell}$, so the total from the first case is overall subexponential. In the second case, we have $c_q > c_{q-1}$ so the leading asymptotic term of the sum is when $a_k$ is maximal, in which case we have $a_k = \alpha\ell+o(\ell)$, which implies $b_k = (1-\alpha)\ell + o(\ell)$. Moreover, we have $\ell_k \leq t$. Hence, we can summarize the total bound in both cases as
\[\# \cK_q(\NN, \ell) \leq \left(c_q^{o(\ell)} \right)+ \left( \ell c_q^{o(\ell)} t^2 q^{2t} c_q^{\alpha\ell + o(\ell)} c_{q-1}^{(1-\alpha)\ell+o(\ell)} \right),\]
where the first summand is the bound from Case 1 and the second summand is the bound from Case 2. This is a bound of 
\[\# \cK(\NN, \ell) \leq q^{2t} c_q^{\alpha \ell + o(\ell)} c_{q-1}^{(1-\alpha)\ell + o(\ell)}.\]
Letting $\eps$ go to $0$ gives the desired bound. 
\end{proof}

\Cref{fig:grow-fac} depicts the relationship between the ratio $f/m = q+1-\alpha$ and the growth rate $\left( \# \cK_q(f,\ell) \right)^{1/m}$ as $m$ grows large. The curve is exponential piecewise, and the growth rate approaches $f/(2m) + 1$ as $f/m$ grows large. The growth rate for all $f/m > 2$ is new.

\begin{figure}[h!]
\begin{tikzpicture}[scale=1.25]
  \draw[->, thick] (0, 0) -- (9, 0) node[right] {\small $f/m$};
  \draw[->, thick] (0, 0) -- (0, 5.5) node[above] {\small $\left(\# \mathcal{K}(f,\ell)\right)^{1/m}$};
  \draw[domain=0:1, smooth, variable=\x, black, line width = 2pt] plot( {\x}, {0});
  \draw[domain=1:2, smooth, variable=\x, red, line width = 2pt] plot( {\x}, {2^(\x-1)});
  \draw[domain=2:4, smooth, variable=\x, orange, line width = 2pt, ] plot( {\x}, {2^(3-\x)*6^(0.5*(\x-2))});
  \draw[domain=4:6, smooth, variable=\x, green, line width = 2pt, ] plot( {\x}, {3^(5-\x)*12^(0.5*(\x-4))});
  \draw[domain=6:8, smooth, variable=\x, blue, line width = 2pt, ] plot( {\x}, {4^(7-\x)*20^(0.5*(\x-6))});
  \draw[->, domain=8:9, smooth, variable=\x, violet, line width = 2pt] plot( {\x}, {5^(9-\x)*30^(0.5*(\x-8))});

  \foreach \i in {1,2,4,6,8} {
  \draw[dotted, thick] (\i,0)--(\i,5.5);
  }

  \foreach \i in {1,...,8} {
  \draw[thick] (\i,0.1)--(\i,-0.1) node[below] {\small \i};
  }
  
  \foreach \i in {1,...,5} {
  \draw[thick] (0.1,\i)--(-0.1,\i) node[left] {\small \i};
  
  }

\end{tikzpicture}

\caption{The relationship between $f/m$ and $\left(\# \cK(f,\ell)\right)^{1/m}$ for large $m$.}
\label{fig:grow-fac}
\end{figure}

\section{Counting numerical semigroups by Frobenius number}
\label{sec:count}

\subsection{General case} In 1990, Backelin~\cite{backelin_1990} established the following asymptotics on $\fr(f)$, the number of numerical semigroups with Frobenius number $f$.

\begin{theorem}[{Backelin \cite[Prop.~1]{backelin_1990}}]
  \label{thm:back1}
  The limits 
  \[C_0 = \lim_{\substack{f\to\infty \\ f \text{ even}}} 2^{-f/2} \fr(f), \quad C_1 = \lim_{\substack{f\to\infty \\ f \text{ odd}}} 2^{-f/2} \fr(f)\]
  exist and are constants.
\end{theorem}

The proof uses the estimate that there are at most $3 \cdot 2^{f/2}$ numerical semigroups with $m \geq f/4$, namely with depth at most $4$. In this section, we obtain an estimate akin to Backelin's result for numerical semigroups of any depth by proving~\Cref{cor:q-depth}, which coupled with~\Cref{prop:2-3-count} implies a strengthening of~\Cref{thm:back0}.

\begin{proof}[Proof of~\Cref{cor:q-depth}]

  The result holds for $q=2$ and $q=3$ by~\Cref{prop:2-3-count}, so assume $q \geq 4$.

  Let $f$ be sufficiently large. \Cref{thm:frob-count} implies the number of numerical semigroups with Frobenius number $f$ and multiplicity $m$ with $(q-1)m < f \leq qm-1$ is 
  \[\left(c_q^{\alpha} c_{q-1}^{1-\alpha}\right)^{m+o(m)} =  \left(c_q^{\alpha} c_{q-1}^{1-\alpha}\right)^{f/(q-1+\alpha) + o(f)},\]
  where $\alpha = f/m - (q-1)$. Here, $m$ varies over the integers in the interval $\left[ \frac{f+1}{q}, \frac{f}{q-1} \right)$. Then by~\Cref{cor:cq-dec}, the right-hand side is asymptotically maximal when $\alpha$ is smallest, so when $m$ is maximal, in which case $m = \floor{(f-1)/(q-1)}$. Let $\ell := m-1$ and $\beta := f/\floor{(f-1)/(q-1)} - (q-1)$. Formally, if we let $\ell := m-1$ we have
  \[
  \begin{aligned}
    \# \cK_q(f,\NN) &= \sum_{m \in \left[ \frac{f+1}{q}, \frac{f}{q-1} \right)} \# \cK_q(f,\ell)
    = \sum_{m \in [ \frac{f+1}{q}, \frac{f}{q-1})} \left( c_q^{\alpha} c_{q-1}^{1-\alpha} \right)^{f/(q-1+\alpha) + o(f)} \\
    &\leq \sum_{m \in [ \frac{f+1}{q}, \frac{f}{q-1})} \left( c_q^{\beta} c_{q-1}^{1-\beta} \right)^{f/(q-1+\beta) + o(f)} \\
    &= \left( c_q^{\beta} c_{q-1}^{1-\beta} \right)^{f/(q-1+\beta) + o(f)},
\end{aligned}\]
by~\Cref{cor:cq-dec}, which shows the upper bound. We also have
  \[
  \begin{aligned}
    \# \cK_q(f,\NN) &= \sum_{m \in [ \frac{f+1}{q}, \frac{f}{q-1})} \left( c_q^{\alpha} c_{q-1}^{1-\alpha} \right)^{f/(q-1+\alpha) + o(f)} \\
    &\geq \left( c_q^{\alpha} c_{q-1}^{1-\alpha} \right)^{f/(q-1+\alpha) + o(f)}  \bigg\rvert_{m = \floor{(f-1)/(q-1)}}
    = \left( c_q^{\beta} c_{q-1}^{1-\beta} \right)^{f/(q-1+\beta) + o(f)},
\end{aligned}\]
which shows the lower bound. We have
\[\beta = \frac{f}{\floor{\frac{f-1}{q-1}}} - (q-1) = \frac{f - (q-1) \floor{\frac{f-1}{q-1}}}{\floor{\frac{f-1}{q-1}}} \leq \frac{q-2}{\floor{\frac{f-1}{q-1}}},\]
so as $f$ grows large, $\beta$ tends to $0$. Hence, $\# \cK_q(f,\NN) = \left(c_{q-1}\right)^{f/(q-1)+o(f)}$, as desired.
\end{proof}

From this result, we can prove that almost all numerical semigroups with Frobenius number $f$ are depth $2$ or $3$. This was first observed by Singhal~\cite[Cor.~1]{singhal_2022}, who derived the result as a corollary of Backelin's work~\cite[Prop.~2]{backelin_1990}.

\begin{corollary}
  \label{cor:small-depth}
  We have $\fr(f) \sim \fr_2(f) + \fr_3(f)$.
\end{corollary}

\begin{proof}
  Evidently $\fr_1(f) = 1$ and $f \geq q$. We cannot directly sum the bound in~\Cref{cor:q-depth} over all $q \in [2,f]$ since the $o(f)$ term in the exponent is non-uniform as $q$ varies. However, we can use the uniform bound in~\Cref{cor:stupid-f}:
  \begin{align*}
    \fr(f) &= \sum_{q = 2}^{f} \fr_q(f) \leq \sum_{q=2}^{100} \fr_q(f) + \sum_{q=101}^{f} \fr_q(f) \\
    &\leq \sum_{q=2}^{100} \fr_q(f) + \sum_{q=101}^f f \cdot q^{f/(q-1)} \\
    &= \sum_{q=2}^{100} \left( c_q \right)^{f/(q-1)+o(f)} + \sum_{q=101}^f f \cdot q^{f/(q-1)}. 
  \end{align*}
 By~\Cref{lem:annoying} with $r = 1$, the first summand has asymptotically leading terms at $q=2$ and $q=3$, which are $2^{f/2 + o(f)}$. These terms also clearly dominate the second summand for large $f$, which yields the result.
\end{proof}

\Cref{prop:2-3-count} and \Cref{cor:small-depth} collectively imply~\Cref{thm:back0}.

\begin{proof}[Proof of \Cref{thm:back0}]
  By~\Cref{cor:small-depth}, almost all numerical semigroups have depth 2 or 3. There are $2^{f/2 + O(1)}$ numerical semigroups with these depths by~\Cref{prop:2-3-count}, as desired. To show the constant term depends on the parity, see~\Cref{cor:back-const}.
\end{proof}

In fact, by summing over all possible $\ell$, we can write down the exact values of $C_0$ and $C_1$ in terms of infinite sums, lifting a result of Backelin~\cite[eq.~30]{backelin_1990} to the language of Kunz words.

\begin{corollary}
  \label{cor:back-const}
  We have that
  \[C_0 = \frac12 + \frac1{2} \sum_{j \text{ even}} \# \cK_3( \ah, j) 2^{-3j/2}, \quad C_1 = \frac1{\sqrt{2}} + \frac12 \sum_{j \text{ odd}} \# \cK_3(\ah, j) 2^{-3j/2}. \]
\end{corollary}

\begin{proof}
  We simply divide the counts derived in~\Cref{prop:2-3-count} by the total from~\Cref{thm:back0}.
\end{proof}

Moreover, \Cref{thm:frob-count} implies \Cref{cor:q-depth-mult}, which determines sharp asymptotics on the number $\mult_q(m)$ of depth-$q$ numerical semigroups with multiplicity $m$.

\begin{proof}[Proof of \Cref{cor:q-depth-mult}]
  The result quickly follows from \Cref{thm:frob-count}, since
  \begin{align*}
    \mult_q(m) &= \#\cK_q(\NN, m-1) \\
    &= \sum_{(q-1)m < f < qm} \#\cK_q(f,m-1) 
    = \sum_{(q-1)m < f < qm} \left( c_q^{\alpha} c_{q-1}^{1-\alpha} \right)^{m + o(m)} \\
    &= c_q^{m + o(m)},
  \end{align*}
  since $\alpha = f/m - (q-1)$ ranges from $1/m$ to $(m-1)/m$ and the summand is asymptotically maximized at $\alpha = (m-1)/m$.
\end{proof}

\subsection{MED semigroups}

A numerical semigroup $\Lambda$ has \vocab{maximal embedding dimension} (abbreviated MED) if its Ap\'ery set minimally generates $\Lambda$. In this case, we say $\Lambda$ is an \vocab{MED semigroup}. In terms of Kunz words, the Kunz word $w_1 \cdots w_{\ell}$ is an MED Kunz word if $w_i + w_j > w_{i+j}$ and $w_{i} + w_j + 1 > w_{\ell+1-i-j}$. 

Let $\MED(f)$ denote the number of MED semigroups with Frobenius number $f$, and $\MED_q(f)$ the number of semigroups with Frobenius number $f$ and depth $q$. In 2022, Singhal~\cite{singhal_2022} proved that there are constants $c$ and $c'$ such that
\[c \cdot 2^{f/3} < \MED(f) < c' \cdot 2^{0.41385f},\]
and conjectured that $\lim_{f\to\infty} \log_2(\MED(f))/f$ exists. He also notes that numerical data suggests that the limit may be around $0.375$. We prove this conjecture by showing this limit to be exactly $1/3$ in \Cref{thm:med-count}, which turns out to be significantly smaller than the value suggested by numerical calculations for small values of $f$.

Given an MED Kunz word $w_1 \cdots w_{\ell}$, if we set $v_i = w_i - 1$, then $v_1 \cdots v_{\ell}$ also satisfies the Kunz conditions. It follows quickly that there is a bijection between MED semigroups of multiplicity $m$ and numerical semigroups containing $m$ (or their corresponding Kunz words); we mark down the latter condition by using the notation $\fr^{\{m\}}$ and $\cK^{\{m\}}$ in lieu of $\fr$ and $\cK$. As a result, we have the following sum.

\begin{lemma}[{\cite[Cor.~3]{singhal_2022}}]
  \label{lem:med-cnt}
  We have that \[\MED_q(f) = \sum_{m \in M_{f,q}} \fr^{\{m\}} (f-m),\]
  where $M_{f,q} := \left[\frac{f+1}{q}, \frac{f+1}{q-1}\right) \cap \NN$.
\end{lemma}

Note that the numerical semigroups counted on the right-hand side have depth $q'$, which is at most $\ceil{(f+1-m)/m} = q-1$, since the multiplicity of these semigroups is at most $m$. Unfortunately, the notation in this section is a bit confusing since we juggle MED semigroups and words in $\cK^{ \{m\}}(f-m)$. We will consistently refer to the multiplicity of the MED semigroup by $m$ and the multiplicity of the corresponding depth-$q'$ word by $m' = \ell + 1$, where $\ell$ is the length of the word. We try to remind the reader of this when possible. 

We bound the number of MED semigroups for all depths $q$ and show that almost all MED semigroups are depth $3$ or $4$. To do this, we use \Cref{lem:med-cnt}, which allows us to apply our bounds for general numerical semigroups to MED numerical semigroups, lifting~\Cref{thm:frob-count} to \Cref{thm:med-count}. This is most important in~\Cref{lem:4-med}, where we estimate the number of MED semigroups with depth at least 4. The case of $q=1$ is vacuous, so we first consider $q=2$.

\begin{proposition}
  \label{prop:2-med}
  We have $\MED_2(f) \sim D_i \cdot 2^{f/4}$ for constants $D_0, \dots, D_3$ where $f \equiv i \pmod{4}$.
\end{proposition}

\begin{proof}
  By~\Cref{lem:med-cnt}, we have 
  \[\MED_2(f) = \sum_{m \in M_{f,2}} \fr^{\{m\}}(f-m) = \sum_{k = 1}^{\floor{(f-1)/2}} \fr(k),\]
  since $m > f-m$. Then we are done; we can express $D_i$ in terms of $C_0$ and $C_1$ by \Cref{cor:back-const}.
\end{proof}

\begin{remark}
  It happens that $\MED_2(f)$ is the sequence \href{http://oeis.org/A103580}{\texttt{A103580}} in the OEIS with each entry repeated twice; the OEIS sequence counts how many nonempty subsets of $\{1,\dots,n\}$ contain all its pairwise sums that are at most $n$, which is exactly the number of nontrivial semigroups with Frobenius number at most $n$. Hence, $\texttt{A103580}(n) = 2^{n/2+O(1)}$, which is (as far as we know) a new result that answers a question of Kitaev on well-based sets on path-schemes~\cite{kitaev_2004}.
\end{remark}

We now estimate the number of MED semigroups of depth at least $3$.

\begin{proposition}
  \label{prop:3-med}
  We have $\MED_3(f) = 2^{f/3 + O(1)}$.
\end{proposition}

\begin{proof} We show the lower and upper bounds separately.

  \paragraph{Lower bound.} Let $m = \ceil{(f+1)/3}$ and abbreviate $n := f - m = 2f/3 + O(1)$. Consider the following families of $2$-Kunz words with Frobenius number $n$:
  \begin{itemize}
    \item $\mathcal{M} = \{ w_1 \cdots w_{(n-4)/2} 21 : w_i \leq 2\}$ if $n$ is even; and
    \item $\mathcal{M} = \{ w_1 \cdots w_{(n-3)/2} 2 : w_i \leq 2\}$ if $n$ is odd.
  \end{itemize}
  These families consist of $2^{(n-4)/2}$ and $2^{(n-3)/2}$ words, respectively. Since the restriction of containing $m$ forces at most one entry in a Kunz word to take the value 2, at least half of the words in $\mathcal{M}$ correspond to semigroups containing $m$, so $\MED_3(f) \geq 2^{(n-4)/2 - 1}$, which is enough.

  \paragraph{Upper bound.} Note that
  \begin{align*}
    \MED_3(f) &= \sum_{m \in M_{f,3}} \fr^{\{m\}}(f-m) \le \sum_{m \in M_{f,3}} \fr (f-m, \NN) \\
    &= \sum_{m \in M_{f,3}} 2^{(f-m)/2 + O(1)} = 2^{f/3 + O(1)},
  \end{align*}
  as desired. 
\end{proof}

\begin{lemma}
  \label{lem:4-med}
  For $q \geq 4$, we have that \[\limsup_{f \to \infty} \MED_q(f)^{1/f} \leq \left(c_{q-2}\right)^{1/(q-1)}.\] 
\end{lemma}

\begin{proof}
  Note that
  \begin{align*}
    \MED_q(f) &= \sum_{m \in M_{f,q}} \fr^{\{m\}} (f-m) \\
    &\leq \sum_{m \in M_{f,q}} \fr^{\{m\}}_{q-1} (f-m) + \sum_{\substack{m \in M_{f,q}\\ q' \geq q}} \fr_{q'} (f-m)\\
    &= \sum_{m \in M_{f,q}} \fr^{\{m\}}_{q-1} (f-m) + \sum_{\substack{m \in M_{f,q}\\ q' \geq q}} (c_{q'-1})^{(f-m)/(q'-1) + o(f)},
  \end{align*}
  the last equality by~\Cref{cor:q-depth}.
  We have the inequality
  \[\frac{f-m}{q'-1} \leq \frac{f - (f+1)/q}{q'-1} < f \cdot \frac{q-1}{q(q'-1)}.\]
  Hence, the second summand can be bounded:
  \[\sum_{\substack{m \in M_{f,q}\\ q' \geq q}} (c_{q'-1})^{(f-m)/(q'-1) + o(f)} \leq \sum_{\substack{m \in M_{f,q}\\ q' \geq q}} \left(c_{q'-1}^{1/(q'-1)}\right)^{f(q-1)/q + o(f)} \leq \sum_{\substack{m \in M_{f,q}}} (c_{q-1})^{f/q + o(f)},\]
  since $c_{q'-1}^{1/(q'-1)}$ is maximized at $q' = q$ by \Cref{lem:annoying} with $r=0$. The lemma also implies that $c_{q-1}^{1/q} < c_{q-2}^{1/(q-1)}$, so it suffices to show the upper bound on the first summand only.

  Letting $\alpha = (f-m)/m' - (q-2)$, we write the following equalities, the second given by~\Cref{thm:frob-count}:
  \begin{align*}
    \sum_{m \in M_{f,q}} \fr_{q-1}^{\{m\}} (f-m) &= \sum_{m \in M_{f,q}} \sum_{m' \leq m} \# \cK_{q-1}^{\{m\}} (f-m;m')\\
    &= \sum_{m \in M_{f,q}} \sum_{m' \leq m} \left( c_{q-1}^{\alpha} c_{q-2}^{1-\alpha} \right)^{m' + o(m')}\\
    &= \sum_{m \in M_{f,q}} \sum_{m' \leq m} \left( c_{q-1}^{f-m-(q-2)m'} c_{q-2}^{(q-1)m'-(f-m)} \right)^{1 + o(1)}.
  \end{align*}
  \Cref{lem:annoying} for $r=0$ tells us that for fixed $m$, the summand is maximized when $m' = m$, at which point the summand is equal to
  \[\left( c_{q-1}^{f-(q-1)m} c_{q-2}^{qm-f} \right)^{1 + o(1)} = \left( c_{q-1}^{f/m-q+1} c_{q-2}^{q-f/m} \right)^{m+o(m)}.\]
  If we let $m = (f+1)/(q-1+r)$, then~\Cref{cor:cq-dec} tells us that the summand is maximized at $r = 0$, at which point the summand is $c_{q-2}^{f/(q-1) + o(f)}$. This is enough to show the upper bound.
\end{proof}

We now derive~\Cref{thm:med-count}. 

\begin{proof}[Proof of~\Cref{thm:med-count}]

  \Cref{prop:3-med} gives the lower bound. As for the upper bound, we cannot sum over all $q$ by the same analytic issue as the previous section, so~\Cref{cor:stupid-f} will be of use to us. Namely, we can write
  \begin{align*}
    \MED(f) &= \sum_{q = 2}^{f} \MED_q(f) = \sum_{q=2}^{100} \MED_q(f) + \sum_{q=101}^{f} \MED_q(f) \\
    &\leq \sum_{q=2}^{100} \MED_q(f) + \sum_{q=101}^f f \cdot q^{f/(q-1)} \\
    &= 2^{f/3+O(1)} + \sum_{q=4}^{100} \left( c_{q-1} \right)^{f/(q-2)+o(f)} + \sum_{q=101}^f f \cdot q^{f/(q-1)}, 
  \end{align*}
  where the last line is given by~\Cref{prop:2-med}, \Cref{prop:3-med}, and \Cref{lem:4-med}. By~\Cref{lem:annoying} for $r=1$, the leading terms are at $q=3$ and $q=4$, which are $2^{f/3 + o(f)}$. These terms also clearly dominate the second summand for large $f$, which yields the result.
\end{proof}

\begin{remark}
  The proof relies on~\Cref{lem:annoying}, and in particular the fact that $c_3^{1/4} = 6^{1/8} < 2^{1/3}$, which is surprisingly sharp given that $2^{1/3} - 6^{1/8} \approx 8.8 \cdot 10^{-3}$. Hence, we esssentially need the full strength of~\Cref{thm:frob-count} for $q=3$; the naive bound from~\Cref{rem:thisishard} and the bound of Backelin~\cite[eq.~13]{backelin_1990} give constants of $8^{1/8}$ and $(22/3)^{1/8}$, respectively, which are certainly not enough.
\end{remark}

It should be possible to write down an analogue of~\Cref{prop:2-3-count} for MED semigroups of depth 3 and 4, which would improve the bound to $2^{f/3 + O(1)}$ and yield a result analogous to~\Cref{cor:back-const} but with constant factor depending on $f \pmod{6}$. However, we do not execute this here.

\section{Distributions on numerical semigroups}
\label{sec:stats}

In this section, we discuss the distribution of multiplicity and genus over numerical semigroups with fixed Frobenius number. 

\subsection{Distribution of multiplicity}

To begin, it turns out that our analysis in~\Cref{prop:2-3-count} immediately implies the following distribution result.

\begin{theorem}
  \label{thm:mult-dist}
  Let $\Lambda_f$ be a random numerical semigroup with Frobenius number $f$ (under the uniform distribution). Then for any integer $k$, we have that:
  \begin{align*}
  \lim_{\substack{f\to \infty \\ f \text{ even}}} \mathbb{P} [ f - 2m(\Lambda_f) = 2k] &= \begin{cases}
    C_0^{-1} 2^{k-1} & \text{if }k < 0, \\
    0 & \text{if }k = 0, \\
    C_0^{-1} 2^{-3k-1} \# \cK_3(\ah; 2k) & \text{if }k > 0;
  \end{cases} \\ \\
  \lim_{\substack{f\to \infty \\ f \text{ odd}} }\mathbb{P} [ f - 2m(\Lambda_f) = 2k+1] &= \begin{cases}
    C_1^{-1}2^{(2k-1)/2} & \text{if }k < 0, \\
    C_1^{-1} 2^{-(6k+5)/2} \# \cK_3(\ah; 2k+1) & \text{if }k \geq 0.
  \end{cases}
\end{align*}
\end{theorem}

This is similar to a result of Zhu~\cite[Thm.~6.1]{zhu_2022}, which determines the limiting distribution of $f-2m$ for numerical semigroups of fixed genus. Figures~\ref{fig:even-f} and \ref{fig:odd-f} depict the limiting distributions for even and odd $f$, respectively.

  \begin{figure}[h!]
    \begin{tikzpicture}[scale=0.8]
    \draw[<->,thick] (-6,0)--(6,0) node[below right] {\small $f-2m$};
    \foreach \hh/\xx [evaluate=\yy using int(2*\xx)] in {0.0625/-5,0.125/-4,0.25/-3,0.5/-2,1/-1,0/0,0.5/1,0.4375/2,0.375/3,0.3256/4,0.2754/5} {
    \filldraw[black] (-0.1+\xx,0)--(0.1+\xx,0)--(0.1+\xx,4*\hh)--(-0.1+\xx,4*\hh)--cycle;
    }
    \foreach \xx in {-10, ..., 10} {
    \node[below] at (\xx/2,0) {\footnotesize \xx};
    }
    \draw[->,thick] (0,0)--(0,4.4);
    \draw (-0.1,4)--(0.1,4) node[right] {\footnotesize $\frac1{4C_0}$};
  \end{tikzpicture}

  \caption{Limiting distribution of $f-2m$ for even $f$.}
    \label{fig:even-f}
\end{figure}

\begin{figure}[h!]
  \begin{tikzpicture}[scale=0.8]
    \draw[<->,thick] (-6,0)--(6,0) node[below right] {\small $f-2m$};
    \foreach \hh/\xx [evaluate=\yy using int(2*\xx)] in {0.0625/-4.5,0.125/-3.5,0.25/-2.5,0.5/-1.5,1/-0.5,0.5/0.5,0.4375/1.5,0.390625/2.5,0.3349609375/3.5,0.2745361328125/4.5} {
    \filldraw[pattern = north east lines, thick] (-0.1+\xx,0)--(0.1+\xx,0)--(0.1+\xx,4*\hh)--(-0.1+\xx,4*\hh)--cycle;
    }
    \draw[->,thick] (0,0)--(0,4.4);
    \foreach \xx in {-10, ..., 10} {
    \node[below] at (\xx/2,0) {\footnotesize \xx};
    }
    \draw (-0.1,4)--(0.1,4) node[right] {\footnotesize $\frac1{2C_1\sqrt{2}}$};
  \end{tikzpicture}

  \caption{Limiting distribution of $f-2m$ for odd $f$.}
  \label{fig:odd-f}
\end{figure}

\begin{proof}
  By~\Cref{cor:small-depth}, we only have to consider semigroups of depth 2 and 3. Now, we may use~\Cref{prop:2-3-count} and the fact that there are $(C_i + o(1))2^{f/2}$ numerical semigroups with Frobenius number $f$ to prove the result, where $f \equiv i \pmod{2}$.
\end{proof}

This is a strengthening of the following result of Backelin, which shows that asymptotically, almost all semigroups have multiplicity close to $f/2$:

\begin{proposition}[Backelin \cite{backelin_1990}]
  For any $\eps > 0$, there exists an integer $N$ such that for every positive integer $f$, there are less than $\eps \cdot 2^{f/2}$ semigroups with Frobenius number $f$ and multiplicity outside of $[f/2 - N, f/2 + N]$.
\end{proposition}

As a corollary, we can show the following result on the average multiplicity of a numerical semigroup with Frobenius number $f$. This in part relies on the fact that $\#\cK_3(\ah, 2k)$ is exponentially less than $8^k$ by \Cref{thm:frob-count}.

\begin{corollary}
  \label{cor:mult-const}
  Let $\Lambda_f$ be a random numerical semigroup with Frobenius number $f$. Then there are constants $\mu_0$ and $\mu_1$ such that as $f$ grows large, the average value of $m(\Lambda_f) - f/2$ approaches $\mu_0$ for $f$ even and $\mu_1$ for $f$ odd.
\end{corollary}

\begin{proof}
  We show the result for $f$ even; the odd case is analogous. By~\Cref{thm:mult-dist}, as $f$ grows large, the average multiplicity tends to
  \begin{align*}
    \sum_{k < 0} C_0^{-1} 2^{k-1} \cdot \left( \frac12 f - k\right) + \sum_{k > 0} C_0^{-1} 2^{-3k-1} \# \cK_3( \ah, 2k) \cdot \left( \frac12 f - k \right) \\
    = \frac{f}{2} + 
    \left(\sum_{k < 0} C_0^{-1} 2^{k-1} \cdot \left(- k\right) + \sum_{k > 0} C_0^{-1} 2^{-3k-1} \# \cK_3( \ah, 2k) \cdot \left(- k \right) \right).
\end{align*}
The right summand consists of the sum of two series. The first is a arithmetico-geometric series with ratio $1/2$, while the second converges since \Cref{cor:q-depth-mult} implies $\#\cK(\ah, 2k) = 6^{k + o(k)}$. Hence, the right summand converges to some constant $\mu_0$, as desired.
\end{proof}

While we cannot prove a fine distribution result for semigroups of arbitrary depth, we can still show almost all depth-$q$ numerical semigroups with Frobenius number $f$ have multiplicity close to $f/(q-1)$.

\begin{theorem}
  \label{prop:q-dist}
  For $q \geq 3$, let $\Lambda_{f,q}$ be a random depth-$q$ numerical semigroup with Frobenius number $f$ (under the uniform distribution). Then, for any $\eps > 0$, we have
  \[\lim_{f\to\infty}\mathbb{P} ( | f - (q-1)\cdot m(\Lambda_{f,q})| < \eps f) = 1.\]
\end{theorem}

\begin{proof}
  We will show the proportion of numerical semigroups with $f-(q-1)m \geq \eps f$ approaches $0$. Recall from \Cref{thm:frob-count} that the number of numerical semigroups with Frobenius number $f$ and multiplicity $m$ is 
  \[\left(c_q^{\alpha} c_{q-1}^{1-\alpha}\right)^{m+o(m)} =  \left(c_q^{\alpha} c_{q-1}^{1-\alpha}\right)^{f/(q-1+\alpha) + o(f)},\]
  where $\alpha = f/m - (q-1) \geq \eps$. By~\Cref{cor:cq-dec}, this quantity is asymptotically maximal when $\alpha = \eps$. Hence, the proportion of numerical semigroups with $f -(q-1)m \geq \eps f$ is
  \[\sum_{f - (q-1)m \geq \eps f} \left(c_q^{\alpha} c_{q-1}^{1-\alpha}\right)^{m+o(m)} \leq \sum_{f - (q-1)m \geq \eps f} \left(c_q^{\eps} c_{q-1}^{1-\eps}\right)^{m+o(m)} = \left(c_q^{\eps} c_{q-1}^{1-\eps}\right)^{f/(q-1+\eps) + o(f)}.\]
  The total number of depth-$q$ numerical semigroups with Frobenius number $f$ is $(c_{q-1})^{f/(q-1) + o(f)}$ by \Cref{thm:frob-count}, which dominates the right-hand side as $f$ grows large. So the proportion we sought is indeed $0$.
\end{proof}

\subsection{Distribution of genus}

Our results on multiplicity can be used to deduce results on the genus of numerical semigroups with fixed Frobenius number. 
We first show the following distribution result on the genus of a numerical semigroup with fixed Frobenius number. The limiting distribution is binomial in the following sense.

\begin{theorem}
  \label{thm:gen-dist}
  Let $\Lambda_f$ be a random numerical semigroup of Frobenius number $f$. Then as $f$ grows large, the distribution of
  \[\frac{1}{\sqrt{f}} \left( g(\Lambda_f) - \frac{3f}{4} \right)\]
  is the standard normal distribution.
\end{theorem}

This strengthens a result of Singhal, who describes the limiting distribution of genus with respect to Frobenius number in terms of a normal distribution and an infinite sum~\cite[Thm.~14]{singhal_2022}.

\begin{proof}

  By~\Cref{cor:small-depth}, we only have to consider depth-2 and depth-3 semigroups, which were characterized in \Cref{prop:2-3-count} and \Cref{thm:mult-dist}. In particular, almost all numerical semigroups of depth $2$ are of the form
  \[\underbrace{w_1 \cdots w_{f-m-1}}_{\text{1's and 2's}} 2 1 \cdots1, \quad m = f/2 + o(\sqrt{f}).\]
  For this family of words, since the last $o(\sqrt{f})$ entries equal to 1 are asymptotically negligible, we have
  \[\frac{1}{\sqrt{f}} \left( g(\Lambda_f) - \frac{3f}{4} \right) = \frac{1}{\sqrt{f}} \left( \sum_{i=1}^{f-m-1} (w_i - 3/2) \right) + o(1),\]
  which is the standard normal distribution by the Central Limit Theorem. Similarly, almost all $3$-Kunz words are of the form
  \[\underbrace{w_1 \cdots w_{f-2m-1}}_{\text{in } \cK_3(\ah; f-2m-1)} 3 \underbrace{w_{f-2m+1} \cdots w_{m-1}}_{\text{1's and 2's}}, \quad m = f/2 - o(\sqrt{f}).\]
  For this family of words, the first $f-2m-1 = o(\sqrt{f})$ entries are also asymptotically negligible, in the sense that we also have
  \[\frac{1}{\sqrt{f}} \left( g(\Lambda_f) - \frac{3f}{4} \right) = \frac{1}{\sqrt{f}} \left( \sum_{i=f-2m+1}^{m-1} (w_i - 3/2) \right) + o(1).\]
  The distribution of this quantity also is standard normal, as desired.
\end{proof}

Moreover, Singhal~\cite[Thm.~1.2]{singhal_2022} shows the average genus of numerical semigroups with Frobenius number $f$ is $\frac{3f}{4} + o(f)$. We strengthen this to an $O(1)$-level estimate, similarly to~\Cref{cor:mult-const}.

\begin{theorem}
  \label{thm:genus-const}
  Let $\Lambda_f$ be a random numerical semigroup with Frobenius number $f$. Then there are constants $\gamma_0$ and $\gamma_1$ such that as $f$ grows large, the average value of $g(\Lambda_f) - \frac{3f}{4}$ approaches $\gamma_0$ for $f$ even and $\gamma_1$ for $f$ odd. 
\end{theorem}

\begin{proof}
  We show the result for $f$ even; the odd case is similar. By~\Cref{cor:small-depth}, we only have to consider depth-2 and depth-3 semigroups, which were characterized in \Cref{prop:2-3-count} and \Cref{thm:mult-dist}. 

  First, we consider the numerical semigroups with depth 2, which have Kunz words of the form $w_1 \cdots w_{f-m-1} 2 1\cdots 1$, where $w_1, \dots, w_{f-m-1} \in [2]$ and $m > f/2$. This family of words has an average genus of $(f-m-1) \cdot \frac32 + 2 + (2m-f-1) \cdot 1 = \frac12 (f+m-1)$. By~\Cref{thm:mult-dist}, the proportion of numerical semigroups with Frobenius number $f$ that have multiplicity $m = f/2 - k$ is $C_0^{-1} 2^{k-1}$ for $k < 0$. Thus, as $f$ grows large, the depth-2 contribution to the overall average genus approaches
  \[\sum_{k<0} C_0^{-1} 2^{k-1} \cdot \left( \frac{3f-2k-2}{4} \right).\]
  
  Similarly, depth-3 semigroups have Kunz words of the form $w_1 \dots w_{m-1}$, where:
  \begin{itemize}
    \item $ w_1 \cdots w_{f-2m-1}w_{f-2m} \in \cK_3(\ah; f-2m-1)$, and 
    \item $w_{f-2m+1}, \dots, w_{m-1} \in [2]$.
  \end{itemize}
  Let $G_j := \frac{1}{\# \cK_3(\ah,j)}\sum_{W \in \cK_3(\ah,j)} g(W)$ denote the average genus of the Kunz words in $\cK_3(\ah,j)$. Then for fixed $f$ and $m$, the average genus of a numerical semigroup with Frobenius number $f$ and multiplicity $m$ is $G_{f-2m} + (3m-f-1) \cdot \frac32$. Letting $k = \frac12 (f-2m)$, these semigroups comprise $C_0^{-1} 2^{-3k-1} \# \cK_3(\ah; 2k)$ of all semigroups with Frobenius number $f$, so the depth-3 contribution to the overall average genus approaches
  \[\sum_{k > 0} C_0^{-1} 2^{-3k-1} \# \cK_3(\ah; 2k) \cdot \left( \frac{2G_{2k} + 3f - 18k - 6}{4} \right). \]
  Hence, the overall average genus is
  \begin{align*}
    \sum_{k<0} C_0^{-1} 2^{k-1} \cdot \left( \frac{3f-2k-2}{4} \right) +  
    \sum_{k > 0} C_0^{-1} 2^{-3k-1} \# \cK_3(\ah; 2k) \cdot \left( \frac{2G_{2k} + 3f - 18k - 6}{4} \right) \\
    = \frac{3f}{4} + \left(\sum_{k<0} C_0^{-1}2^{k-1} \left( \frac{-k-1}{2} \right) + \sum_{k>0} C_0^{-1} 2^{-3k-1} \# \cK_3(\ah; 2k) \cdot \left( \frac{2G_{2k}-18k-6}{4} \right)\right)
  \end{align*}
  It suffices to show that the right summand in parentheses converges to some constant $\gamma_0$. The first summation clearly converges, and the second summation also converges since $G_j \leq 3j$, and $\# \cK_3(\ah; 2k) = 6^{k + o(k)}$ by~\Cref{thm:frob-count} for $q=3$. Hence, we have the desired result.
\end{proof}

\section{Future directions}
\label{sec:direcs}

In this section, we discuss our results and some possible directions of future study.

\subsection{Subexponential factor}

In this paper, we prove many results that pin down the exponential growth factor but do not estimate the subexponential factor. Our bounds are of the form $C^{n + o(n)}$, so it would be nice to sharpen these to a polynomial factor $p(n) C^n$ or even a constant factor $A \cdot C^n$. 

For instance, \Cref{thm:frob-count} implies that $\# \cK_3 ( \ah , \ell) = 6^{\ell/2 + o(\ell)}$, but we write down rough $\ell^{O(\sqrt{\ell})}$-level upper bounds on the $6^{o(\ell)}$ factor while our construction only gives $6^{\ell/2 + O(1)}$ Kunz words. Reconciling these bounds would yield a better understanding of the structure of $3$-Kunz words. \Cref{tab:3-kunz} lists $\# \cK_3 (\ah , \ell)$ for $\ell \leq 56$ due to Zhu~\cite[\S 7.1]{zhu_2022}, while \Cref{fig:3-kunz} graphs the quantity $6^{-\ell/2} \cdot \# \cK_3( \ah, \ell)$. It is unclear whether we should expect the subexponential factor to be constant, linear, or another function altogether.

Since $\# \cK_3(\NN,\ell) = \sum_{j=0}^\ell 2^{\ell-j} \# \cK_3(\ah, j)$, better bounds on $\# \cK_3(\ah , \ell)$ automatically give better bounds on $\#\cK_3 (\NN, \ell)$. Sharpening these bounds could help resolve a question of Kaplan~\cite{kaplan_2017}, who asks whether $\lim_{k \to \infty} \fr(2k)/ \fr(2k-1) = \sqrt{2} \cdot C_0/ C_1$ is greater than, less than, or equal to 1. Backelin \cite[eq.~13]{backelin_1990} proves the weaker but computationally useful bound
\[\# \cK_3 (\ah, \ell) \leq 2^{\left\lfloor{\frac{3\ell-3}{2}}\right\rfloor} \left( \frac{11}{12} \right)^{\left\lfloor{\frac{\ell-1}{2}}\right\rfloor}.\]
\Cref{cor:back-const} expresses $C_i$ as an infinite sum. We can take partial sums using~\Cref{tab:3-kunz} to get a lower bound on $C_i$, and also use Backelin's bound to get an upper bound on $C_i$. This method gives
\[1.2606 < C_0 < 1.3919, \quad 1.2755 < C_1/\sqrt{2} < 1.4068,\]
which are a modest improvement to Backelin's original bounds\footnote{Backelin originally proves $1.235 < C_0 < 1.65$ and $1.25 < C_1/\sqrt{2} < 1.66$; see pg.~210 (their constants differ from ours by a factor of 2).}~\cite[\S I.3]{backelin_1990}. The numerics of the partial sums suggest that $C_0 < C_1/\sqrt{2}$, so we surmise that Kaplan's limit is less than $1$.

Of course, the larger question of finding the subexponential factor of $\# \cK_q(\NN , \ell)$ or $\fr_q(f)$ for general $q$ is also open. Results of this form could strengthen the asymptotics of~\Cref{thm:back0} to lower-order terms, and could sharpen~\Cref{prop:q-dist} to an $O(1)$-level distribution like~\Cref{thm:mult-dist}. To get a sense of the magnitude of the subexponential factor in the general case, we have included the values of $\#\cK(f,m-1)$ for small values of $f$ and $m$ in~\Cref{fig:small-fm}, and we have graphed the values of $\left(\#\cK(f,m-1)\right)^{1/m}$ against $f/m$ for these values in~\Cref{fig:graph2}. The values for $m = 8$, $10$, $12$ were acquired with the \texttt{numericalsgps} GAP package, while the values for $m=15$ were calculated with a program similar to Zhu's program to compute $\#\cK_3(\ah, \ell)$~\cite[\S 7.1]{zhu_2022}.

\subsection{Other statistics} There are many classes of numerical semigroups other than MED semigroups that remain to be counted. Backelin shows that the number of \vocab{irreducible semigroups}, or semigroups that are not the intersection of two other semigroups, with Frobenius number $f$ is $2^{f/6 + O(1)}$. These have a relatively natural description in terms of Kunz words, so our methods may adapt to questions of this type. There are similar parametrizations in terms of Kunz words for Arf~\cite{arf_2017} and atomic~\cite{rpr_2019} numerical semigroups, which could also be counted. We also do not compute the subexponential term for the number of MED numerical semigroups; we expect that, for fixed $q$, one could obtain sharper bounds for $\MED_q(f)$ akin to~\Cref{cor:q-depth} and \Cref{prop:2-3-count}.

We do not thoroughly study the genus of numerical semigroups in this paper. Zhu~\cite{zhu_2022} shows the number of semigroups with genus $g$ and depth at least $4$ has exponential growth rate between 1.51 and 1.55 and conjectures the true growth rate to be the unique positive zero of $x^6 - x^3 - 2x^2 - 2x - 1$ with value near $1.51519$. Our methods work well for tail-heavy words, but there seem to be fundamental obstructions to our estimates when a $q$-Kunz word has its rightmost heavy chunk near the middle of the word or has no heavy chunks at all. Results in this direction could shed light on the remaining size conjectures of Bras-Amor\'os~\cite{bras-amoros_2007}; for more info, see the final section of~\cite{zhu_2022}.

\subsection{Polychromatic Schur problems}

In this paper, we bound the size of $\cK_3(\NN, \ell)$, which is the set of words $w_1 \cdots w_\ell$ such that there are no $i,j$ with $w_i = w_j = 1$ and $w_{i+j} = 3$. We can rephrase this suggestively as an additive combinatorics question about colorings.

\begin{question}
  How many ways can $\{1,2,\dots,n\}$ be colored red, green, and blue so that a red number plus a red number is never a blue number?
\end{question}

Recall that $x,y,z \in [n]$ (not necessarily distinct) form a \emph{Schur triple} if $x+y=z$. Then the above question is equivalent to avoiding a Schur triple with $x$, $y$ red and $z$ blue. \Cref{cor:q-depth-mult} for $q=3$ implies there are $6^{n/2 + o(n)}$ such colorings. Many questions can be rephrased in terms of colorings and Schur triples:
\begin{itemize}
  \item With two colors, black and white, sum-free sets are equivalent to $\textsf{black} + \textsf{black} \neq \textsf{black}$ and numerical semigroups are equivalent to $\textsf{black} + \textsf{black} \neq \textsf{white}$.
  \item With $k$ colors and sufficiently large $n$, we cannot color $[n]$ to avoid $\textsf{color} + \textsf{color} = \textsf{color}$ for every $\textsf{color}$ by \emph{Schur's theorem}.
  \item With three colors, avoiding $\textsf{red} + \textsf{blue} = \textsf{green}$ and permutations is equivalent to avoiding rainbow Schur triples~\cite{cheng_2020}.
\end{itemize}
The typical language of independent sets and hypergraph containers~\cite{balogh_2019} has interesting applications in the polychromatic case via the language of color templates in~\cite{cheng_2020}. We study the above question in its own right via the lens of graph homomorphisms and hypergraph containers in future work. 

\bibliographystyle{semigroup}
\bibliography{semigroup}

\begin{thebibliography}{10}
\providecommand{\url}[1]{\texttt{#1}}
\providecommand{\urlprefix}{URL }
\providecommand{\eprint}[2][]{\url{#2}}

\bibitem{bacher_2019}
R.~Bacher.
\newblock Generic numerical semigroups (2019).
\newblock {a}rXiv preprint, \eprint[arXiv]{2105.04200v1}.

\bibitem{backelin_1990}
J.~Backelin.
\newblock On the number of semigroups of natural numbers.
\newblock \emph{Mathematica Scandinavica}, \textbf{66} (1990), 197--215.

\bibitem{balogh_2019}
J.~Balogh, R.~Morris, and W.~Samotij.
\newblock The method of hypergraph containers.
\newblock \emph{Proceedings of the International Congress of Mathematicians
  (ICM 2018)}, \textbf{4} (2019), 3077--3110.

\bibitem{blanco_rosales_2012}
V.~Blanco and J.~Rosales.
\newblock On the enumeration of the set of numerical semigroups with fixed
  {F}robenius number.
\newblock \emph{Computers \& Mathematics with Applications}, \textbf{63}(7)
  (2012), 1204–1211.

\bibitem{bor_2021}
M.~B. Branco, I.~Ojeda, and J.~C. Rosales.
\newblock The set of numerical semigroups of a given multiplicity and
  {F}robenius number.
\newblock \emph{Portugaliae Mathematica}, \textbf{78}(2) (2021), 147–167.

\bibitem{bras-amoros_2007}
M.~Bras-Amor\'os.
\newblock Fibonacci-like behavior of the number of numerical semigroups of a
  given genus.
\newblock \emph{Semigroup Forum}, \textbf{76}(2) (2007), 379–384.

\bibitem{cheng_2020}
Y.~Cheng, Y.~Jing, L.~Li, G.~Wang, and W.~Zhou.
\newblock Integer colorings with forbidden rainbow sums (2020).
\newblock {a}rXiv preprint, \eprint{2005.14384v1}.

\bibitem{ef_2020}
S.~Eliahou and J.~Fromentin.
\newblock Gapsets and numerical semigroups.
\newblock \emph{Journal of Combinatorial Theory, Series A}, \textbf{169}
  (2020), 105--129.

\bibitem{arf_2017}
P.~A. Garc\'ia-S\'anchez, B.~A. Heredia, H.~I. Karakas, and J.~C. Rosales.
\newblock Parametrizing {Arf} numerical semigroups.
\newblock \emph{Journal of Algebra and Its Applications}, \textbf{16}(11)
  (2017), 175--209.

\bibitem{kaplan_2011}
N.~Kaplan.
\newblock Counting numerical semigroups by genus and some cases of a question
  of {W}ilf.
\newblock \emph{Journal of Pure and Applied Algebra}, \textbf{216}(5) (2012),
  1016--1032.
\newblock ISSN 0022-4049.

\bibitem{kaplan_2017}
N.~Kaplan.
\newblock Counting numerical semigroups.
\newblock \emph{The American Mathematical Monthly}, \textbf{124}(9) (2017),
  862--875.

\bibitem{kitaev_2004}
S.~Kitaev.
\newblock Independent sets on path-schemes.
\newblock \emph{Journal of Integer Sequences}, \textbf{9}(06.2.2) (2006), 1--8.

\bibitem{kunz_1987}
E.~Kunz.
\newblock \emph{\"Uber die Klassifikation Numerischer Halbgruppen}.
\newblock Fakult\"at fur Mathematik der Universit\"at (1987).

\bibitem{rpr_2019}
A.~M. Robles-P\'erez and J.~C. Rosales.
\newblock The enumeration of the set of atomic numerical semigroups with fixed
  {Frobenius} number.
\newblock \emph{Journal of Algebra and Its Applications}, \textbf{19}(08)
  (2019), 2050144.

\bibitem{rosales_2009}
J.~Rosales and P.~Garc\'ia-S\'anchez.
\newblock \emph{Numerical semigroups}.
\newblock Springer-Verlag (2009).

\bibitem{singhal_2022}
D.~Singhal.
\newblock Distribution of genus among numerical semigroups with fixed
  {F}robenius number.
\newblock \emph{Semigroup Forum}, \textbf{104} (2022), 704--723.

\bibitem{zhai_2012}
A.~Zhai.
\newblock Fibonacci-like growth of numerical semigroups of a given genus.
\newblock \emph{Semigroup Forum}, \textbf{86}(3) (2012), 634–662.

\bibitem{zhao_2011}
Y.~Zhao.
\newblock The bipartite swapping trick on graph homomorphisms.
\newblock \emph{SIAM Journal on Discrete Mathematics}, \textbf{25}(2) (2011),
  660–680.

\bibitem{zhu_2022}
D.~G. Zhu.
\newblock Sub-{F}ibonacci behavior in numerical semigroup enumeration (2022).
\newblock {a}rXiv preprint, \eprint[arXiv]{2022.05755v1}.

\end{thebibliography}

\newpage
\appendix

\section{Sample data on $\#\cK_q(f,\ell)$}

\begin{table}[h!]
  {\footnotesize
  \begin{tabular}{|rr|rr|rr|rr|}
\hline
$\ell$ & $\#\cK_3(\ah,\ell)$ & $\ell$ & $\#\cK_3(\ah,\ell)$ & $\ell$ & $\#\cK_3(\ah,\ell)$ & $\ell$ & $\#\cK_3(\ah,\ell)$ \\ \hline
1 & 1 & 15 & 636988 & 29 & 269006491243 & 43 & 93330716828074728 \\
2 & 2 & 16 & 1264258 & 30 & 514570562660 & 44 & 182169304991649599 \\
3 & 7 & 17 & 4215132 & 31 & 1675924761549 & 45 & 563400466799404781 \\
4 & 14 & 18 & 8051166 & 32 & 3260563309970 & 46 & 1123123932176762798 \\
5 & 50 & 19 & 26991332 & 33 & 10226788893396 & 47 & 3515289384328363733 \\
6 & 96 & 20 & 52219388 & 34 & 20391731774615 & 48 & 6748185987886118499 \\
7 & 343 & 21 & 167869363 & 35 & 64492588219388 & 49 & 21482645364583893141 \\
8 & 667 & 22 & 335811042 & 36 & 123297229488909 & 50 & 42180209153883948485 \\
9 & 2249 & 23 & 1088912364 & 37 & 399014138303512 & 51 & 129644921982559989678 \\
10 & 4513 & 24 & 2061900838 & 38 & 783212435011160 & 52 & 256751992776208115484 \\
11 & 15349 & 25 & 6827159829 & 39 & 2425228785559883 & 53 & 803119580525790882344 \\
12 & 28897 & 26 & 13424984452 & 40 & 4789418078046239 & 54 & 1553823126150392917494 \\
13 & 100425 & 27 & 42195919228 & 41 & 15182994877727803 & 55 & 4893440472071899127094 \\
14 & 197268 & 28 & 83374340587 & 42 & 29235444078764327 & 56 & 9583422277969715823101 \\ \hline
\end{tabular}
}
\medskip

\caption{Table of $\#\cK_3(\ah,\ell)$ for $\ell \leq 56$, due to Zhu~\cite[\S 7.1]{zhu_2022}.}
\label{tab:3-kunz}
\end{table}

\begin{figure}[h!]
\begin{tikzpicture}[scale=1.6]
  \draw[->,thick] (0,0) -- (8.8,0) node[right] {\footnotesize $\ell$};
  \draw[->,thick] (0,0) -- (0,5*6/5) node[above] {\footnotesize $ 6^{-\ell/2}\cdot \# \cK_3(\ah, \ell)$};
  \foreach \xx/\yy in {
1/0.408248290463863,2/0.333333333333333,3/0.47628967220784,4/0.388888888888889,5/0.567011514533143,6/0.444444444444445,7/0.648283164949561,8/0.514660493827161,9/0.708449386769467,10/0.580375514403293,11/0.80583886449715,12/0.619362997256517,13/0.878736594860972,14/0.704689643347052,15/0.928959698095263,16/0.752706571025759,17/1.02453205558861,18/0.798909393575676,19/1.09342107028656,20/0.863613204182121,21/1.13340027432871,22/0.925616777882566,23/1.22533136285418,24/0.947224168397517,25/1.28041112922409,26/1.02789211932794,27/1.31895072867331,28/1.06393453626153,29/1.40142425355422,30/1.09440024573313,31/1.45515795121664,32/1.1557732300158,33/1.47993847089379,34/1.20471093786652,35/1.5554746034867,36/1.21403389223042,37/1.60394751215938,38/1.28530387484685,39/1.62481280051669,40/1.30995911598507,41/1.69534001454467,42/1.33270318915486,43/1.73689161349055,44/1.38403697622664,45/1.74748746371675,46/1.42216151717488,47/1.81721661365418,48/1.42415426409653,49/1.85089645989517,50/1.48363636377793,51/1.86165253447247,52/1.50515530306925,53/1.92208274172569,54/1.51816096476897,55/1.95188813453195,56/1.56057848181608} {
  \node[vtx] at (\xx / 50 * 7.7 ,\yy / 1.85 * 4.4 *6/5 ) {};
  }

  \foreach \xx [evaluate=\yy using int(\xx*10)] in {1,...,5} {
  \draw ( \xx / 5 * 7.7, 0.1) -- (\xx / 5 * 7.7, -0.1) node[below] {\footnotesize \yy};
  }

  \draw (0.1,0.5*4.4/1.85*6/5)--(-0.1,0.5*4.4/1.85*6/5) node[left] {\footnotesize 0.5};
  \draw (0.1,1*4.4/1.85*6/5)--(-0.1,1*4.4/1.85*6/5) node[left] {\footnotesize 1.0};
  \draw (0.1,1.5*4.4/1.85*6/5)--(-0.1,1.5*4.4/1.85*6/5) node[left] {\footnotesize 1.5};
  \draw (0.1,2*4.4/1.85*6/5)--(-0.1,2*4.4/1.85*6/5) node[left] {\footnotesize 2.0};
\end{tikzpicture}
\caption{Graph of $6^{-\ell/2}\cdot \# \cK_3(\ah, \ell)$ versus $\ell$ for $\ell \leq 56$.}
\label{fig:3-kunz}
\end{figure}

\begin{table}[]
  {\tiny
\begin{tabular}{|r|r|r|r|r|}\hline
  \diagbox[height=1.5em]{$f$}{$m$}& 8 & 10 & 12 & 15 \\\hline
1 & 0 & 0 & 0 & 0 \\
2 & 0 & 0 & 0 & 0 \\
3 & 0 & 0 & 0 & 0 \\
4 & 0 & 0 & 0 & 0 \\
5 & 0 & 0 & 0 & 0 \\
6 & 0 & 0 & 0 & 0 \\
7 & 0 & 0 & 0 & 0 \\
8 & 0 & 0 & 0 & 0 \\
9 & 1 & 0 & 0 & 0 \\
10 & 2 & 0 & 0 & 0 \\
11 & 4 & 1 & 0 & 0 \\
12 & 8 & 2 & 0 & 0 \\
13 & 16 & 4 & 1 & 0 \\
14 & 32 & 8 & 2 & 0 \\
15 & 64 & 16 & 4 & 0 \\
16 & 0 & 32 & 8 & 1 \\
17 & 64 & 64 & 16 & 2 \\
18 & 64 & 128 & 32 & 4 \\
19 & 112 & 256 & 64 & 8 \\
20 & 112 & 0 & 128 & 16 \\\hline
\end{tabular}
\begin{tabular}{|r|r|r|r|r|}\hline
  \diagbox[height=1.5em]{$f$}{$m$}& 8 & 10 & 12 & 15 \\\hline
21 & 200 & 256 & 256 & 32 \\
22 & 192 & 256 & 512 & 64 \\
23 & 343 & 448 & 1024 & 128 \\
24 & 0 & 448 & 0 & 256 \\
25 & 382 & 800 & 1024 & 512 \\
26 & 334 & 768 & 1024 & 1024 \\
27 & 561 & 1372 & 1792 & 2048 \\
28 & 450 & 1334 & 1792 & 4096 \\
29 & 850 & 2249 & 3200 & 8192 \\
30 & 676 & 0 & 3072 & 0 \\
31 & 1210 & 2664 & 5488 & 8192 \\
32 & 0 & 2223 & 5336 & 8192 \\
33 & 1285 & 3819 & 8996 & 14336 \\
34 & 1093 & 3402 & 9026 & 14336 \\
35 & 1810 & 5481 & 15349 & 25600 \\
36 & 1394 & 4703 & 0 & 24576 \\
37 & 2426 & 8471 & 17913 & 43904 \\
38 & 2004 & 6992 & 15891 & 42688 \\
39 & 3251 & 11859 & 25309 & 71968 \\
40 & 0 & 0 & 23316 & 72208 \\\hline
\end{tabular}
\begin{tabular}{|r|r|r|r|r|}\hline
  \diagbox[height=1.5em]{$f$}{$m$}& 8 & 10 & 12 & 15 \\\hline
41 & 3623 & 13081 & 39888 & 122792 \\
42 & 2871 & 10985 & 31378 & 115588 \\
43 & 4758 & 18110 & 57967 & 200850 \\
44 & 3450 & 15139 & 49397 & 197268 \\
45 & 6045 & 23772 & 80208 & 0 \\
46 & 4833 & 20803 & 72222 & 212355 \\
47 & 7956 & 33328 & 122602 & 334256 \\
48 & 0 & 27434 & 0 & 306617 \\
49 & 8441 & 45180 & 134412 & 495251 \\
50 & 6750 & 0 & 118138 & 482970 \\
51 & 10586 & 49330 & 173933 & 697846 \\
52 & 7781 & 40729 & 156928 & 716142 \\
53 & 13340 & 65554 & 255440 & 1078547 \\
54 & 10409 & 53231 & 202025 & 993543 \\
55 & 16749 & 81087 & 346341 & 1570166 \\
56 & 0 & 68858 & 292872 & 1538669 \\
57 & 17766 & 109712 & 441032 & 2193066 \\
58 & 14017 & 90045 & 402170 & 2204659 \\
59 & 22161 & 143912 & 636627 & 3327962 \\
60 & 15625 & 0 & 0 & 0\\ \hline
\end{tabular}
}
\medskip

\caption{Values of $\#\cK(f,m-1)$ for $m = 8$, $10$, $12$, $15$ and $f \leq 60$.}

\label{fig:small-fm}
\end{table}

\begin{figure}[h!]
\begin{tikzpicture}[scale=1.15]
  \draw[->, thick] (0, 0) -- (9, 0) node[right] {\small $f/m$};
  \draw[->, thick] (0, 0) -- (0, 11) node[above] {\small $\left(\# \mathcal{K}(f,m-1)\right)^{1/m}$};
  \draw[domain=0:1, smooth, variable=\x, gray, line width = 2pt] plot( {\x}, {0});
  \draw[domain=1:2, smooth, variable=\x, gray, line width = 2pt] plot( {\x}, {2*2^(\x-1)});
  \draw[domain=2:4, smooth, variable=\x, gray, line width = 2pt, ] plot( {\x}, {2*2^(3-\x)*6^(0.5*(\x-2))});
  \draw[domain=4:6, smooth, variable=\x, gray, line width = 2pt, ] plot( {\x}, {2*3^(5-\x)*12^(0.5*(\x-4))});
  \draw[domain=6:8, smooth, variable=\x, gray, line width = 2pt, ] plot( {\x}, {2*4^(7-\x)*20^(0.5*(\x-6))});
  \draw[->, domain=8:9, smooth, variable=\x, gray, line width = 2pt] plot( {\x}, {2*5^(9-\x)*30^(0.5*(\x-8))});

  \foreach \i in {1,...,8} {
  \draw[thick] (\i,0.1)--(\i,-0.1) node[below] {\small \i};
  }
  
  \foreach \i in {1,...,5} {
  \draw[thick] (0.1,2*\i)--(-0.1,2*\i) node[left] {\small \i};
  
  }

  \foreach \xx/\yy in {
0.125/0,0.25/0,0.375/0,0.5/0,0.625/0,0.75/0,0.875/0,1/0,1.125/1,1.25/1.09050773266526,1.375/1.18920711500272,1.5/1.29683955465101,1.625/1.4142135623731,1.75/1.54221082540794,1.875/1.68179283050743,2/0,2.125/1.68179283050743,2.25/1.68179283050743,2.375/1.80364994480514,2.5/1.80364994480514,2.625/1.93922744748686,2.75/1.92935725992062,2.875/2.07449200303565,3/0,3.125/2.10260610157792,3.25/2.06760849740343,3.375/2.20607542963356,3.5/2.14610795432031,3.625/2.32368559498247,3.75/2.25810086435323,3.875/2.42855627390301,4/0,4.125/2.44688123824024,4.25/2.39788041332331,4.375/2.55393489406298,4.5/2.47191109443687,4.625/2.64917927952036,4.75/2.58664729842216,4.875/2.74790757443717,5/0,5.125/2.78537415641279,5.25/2.70554118320307,5.375/2.88189434430224,5.5/2.76839073353734,5.625/2.96943988874475,5.75/2.88753395180784,5.875/3.0731718526452,6/0,
6.125/3.09598777785354,6.25/3.01066872838416,6.375/3.18486835879759,6.5/3.06463972602426,6.625/3.27826818287817,6.75/3.1781626981401,6.875/3.37286229390605,7/0,7.125/3.39780700389415,7.25/3.29861693398701,7.375/3.4930006722276,7.5/3.34370152488211
} {
\node[smallvtx,color=red] at (\xx,2*\yy) {};
}

  \foreach \xx/\yy in {
0.1/0,0.2/0,0.3/0,0.4/0,0.5/0,0.6/0,0.7/0,0.8/0,0.9/0,1/0,1.1/1,1.2/1.07177346253629,1.3/1.14869835499704,1.4/1.23114441334492,1.5/1.31950791077289,1.6/1.4142135623731,1.7/1.5157165665104,1.8/1.62450479271247,1.9/1.74110112659225,2/0,2.1/1.74110112659225,2.2/1.74110112659225,2.3/1.84131377177951,2.4/1.84131377177951,2.5/1.95123239964689,2.6/1.94328331572615,2.7/2.05937488080307,2.8/2.05359871048903,2.9/2.16370946833878,3/0,3.1/2.20066239386332,3.2/2.16119496097064,3.3/2.28136600536486,3.4/2.25513959860965,3.5/2.36529862980858,3.6/2.32936474365196,3.7/2.47054897849261,3.8/2.42359552517777,3.9/2.55508179079589,4/0,4.1/2.58026371149762,4.2/2.53559569722737,4.3/2.66558076002816,4.4/2.61824160011783,4.5/2.73909170152321,4.6/2.70279182650179,4.7/2.8332240864183,4.8/2.77861843932651,4.9/2.9207515958605,5/0,5.1/2.94653165774012,5.2/2.8906158168537,5.3/3.03151638352959,5.4/2.96904282949157,5.5/3.0966714013954,5.6/3.04645972095696,5.7/3.19172461686035,5.8/3.1292910660725,5.9/3.27951559836426
  } {
  \node[smallvtx, regular polygon, regular polygon sides = 3] at (\xx,2*\yy) {}; 
  }

\foreach \xx/\yy in {
0.0833333333333333/0,0.166666666666667/0,0.25/0,0.333333333333333/0,0.416666666666667/0,0.5/0,0.583333333333333/0,0.666666666666667/0,0.75/0,0.833333333333333/0,0.916666666666667/0,1/0,1.08333333333333/1,1.16666666666667/1.0594630943593,1.25/1.12246204830937,1.33333333333333/1.18920711500272,1.41666666666667/1.25992104989487,1.5/1.33483985417003,1.58333333333333/1.4142135623731,1.66666666666667/1.49830707687668,1.75/1.5874010519682,1.83333333333333/1.68179283050743,1.91666666666667/1.78179743628068,2/0,2.08333333333333/1.78179743628068,2.16666666666667/1.78179743628068,2.25/1.86685892529162,2.33333333333333/1.86685892529162,2.41666666666667/1.95927696034353,2.5/1.95262315155479,2.58333333333333/2.04935804934867,2.66666666666667/2.04456686607827,2.75/2.13552234767715,2.83333333333333/2.13611490693428,2.91666666666667/2.23274953285872,3/0,3.08333333333333/2.26167758484936,3.16666666666667/2.23921572662548,3.25/2.32776740256566,3.33333333333333/2.31191131520047,3.41666666666667/2.41770618641184,3.5/2.36983861395169,3.58333333333333/2.49420264469206,3.66666666666667/2.46117057693528,3.75/2.56262351674284,3.83333333333333/2.54032408556671,3.91666666666667/2.65485901422487,4/0,4.08333333333333/2.67528365969451,4.16666666666667/2.64666597689238,4.25/2.73337047488128,4.33333333333333/2.71003585392481,4.41666666666667/2.82232750886283,4.5/2.76768775907941,4.58333333333333/2.89484506990699,4.66666666666667/2.85467379091678,4.75/2.95374166057296,4.83333333333333/2.9311236475849,
4.91666666666667/3.04548936874046
} {
\node[smallvtx, diamond, color=green!70!black] at (\xx,2*\yy) {};
}

\foreach \xx/\yy in {
0.0666666666666667/0,0.133333333333333/0,0.2/0,0.266666666666667/0,0.333333333333333/0,0.4/0,0.466666666666667/0,0.533333333333333/0,0.6/0,0.666666666666667/0,0.733333333333333/0,0.8/0,0.866666666666667/0,0.933333333333333/0,1/0,1.06666666666667/1,1.13333333333333/1.04729412282063,1.2/1.09682497969463,1.26666666666667/1.14869835499703,1.33333333333333/1.20302503608212,1.4/1.25992104989487,1.46666666666667/1.31950791077289,1.53333333333333/1.38191287996778,1.6/1.44726923744038,1.66666666666667/1.5157165665104,1.73333333333333/1.5874010519682,1.8/1.66247579228558,1.86666666666667/1.74110112659225,1.93333333333333/1.82344497711643,2/0,2.06666666666667/1.82344497711643,2.13333333333333/1.82344497711643,2.2/1.89275847514723,2.26666666666667/1.89275847514723,2.33333333333333/1.96735468723647,2.4/1.96200787663483,2.46666666666667/2.03938993991828,2.53333333333333/2.03557474423504,2.6/2.10770242685581,2.66666666666667/2.10817028576041,2.73333333333333/2.18412758976766,2.8/2.17534185790977,2.86666666666667/2.25696481699308,2.93333333333333/2.25425881238082,3/0,3.06666666666667/2.26536141675006,3.13333333333333/2.33491968962057,3.2/2.32152348805779,3.26666666666667/2.39692819525829,3.33333333333333/2.39291906771735,3.4/2.45235858314497,3.46666666666667/2.45659337512517,3.53333333333333/2.52458078077428,3.6/2.51080187631274,3.66666666666667/2.58858862206258,3.73333333333333/2.58509404324536,3.8/2.64689540569071,3.86666666666667/2.64782591477319,3.93333333333333/2.72152225818849,4/0
} {
\node[smallvtx, star, star points = 5, color=blue!70] at (\xx,2*\yy) {};
}

\node[vtx, color=red] at (10,11) {};

\node[right] at (10.1,11) {\small $m=8$};

\node[vtx, regular polygon, regular polygon sides = 3] at (10,10.5) {};
\node[right] at (10.1,10.5) {\small $m=10$};

\node[vtx, color=green!70!black, diamond] at (10,10) {};
\node[right] at (10.1,10) {\small $m=12$};

\node[vtx, color=blue!70, star, star points = 5] at (10,9.5) {};
\node[right] at (10.1,9.5) {\small $m=15$};
\end{tikzpicture}

\caption{The relation between $f/m$ and $\left(\# \cK(f,m-1)\right)^{1/m}$ for $m=8$, $10$, $12$, $15$.}

\label{fig:graph2}
\end{figure}

\end{document}